\documentclass[12pt]{amsart}
\usepackage{amsmath,amssymb}

\usepackage{enumerate}
\usepackage{graphicx}
\usepackage{overpic}
\usepackage{float}
\usepackage{xcolor}
\usepackage{accents}
\usepackage{hyperref}

\usepackage{mathtools}

\DeclareMathAlphabet{\mathpzc}{OT1}{pzc}{m}{it}

\DeclareMathOperator{\tr}{tr}
\DeclareMathOperator{\cof}{cof}
\DeclareMathOperator{\sgn}{sgn}
\DeclareMathOperator{\spn}{span}

\newcommand{\br}{\bar\varrho}
\newcommand{\hd}{\mathpzc{h}}
\newcommand{\shear}{\mathpzc{g}}
\newcommand{\nullf}{\mathpzc{n}}
\newcommand{\volume}{v}
\newcommand{\tdsf}{a}
\newcommand{\wh}{W}
\newcommand{\pksh}{S}
\newcommand{\zinf}{\|\zeta\|_\infty}
\newcommand{\scp}{\sigma}
\newcommand{\omt}{\Omega_t}
\newcommand{\WT}{\widetilde{\Phi}}

\newcommand{\rr}{\mathbb{R}}
\newcommand{\mm}{\mathbb{M}}
\newcommand{\cc}{\mathcal C}
\newcommand{\bb}{\mathcal B}
\newcommand{\bbbar}{\overline{\mathcal B}}
\newcommand{\ff}{\mathcal F}
\newcommand{\ffzbm}{\mathcal F(\zeta,\mu)}
\newcommand{\llll}{\mathcal L}
\newcommand{\uu}{\mathcal U}
\newcommand{\eps}{\varepsilon}

\newcommand{\kk}{\kappa}
\newcommand{\half}{{\tfrac12}}
\newcommand{\third}{\tfrac13}
\newcommand{\glt}{{\rm GL}_+(3,\rr)}
\newcommand{\slt}{{{\rm SL}(3,\rr)}}
\newcommand{\sot}{{\rm SO}(3,\rr)}

\newcommand{\isp}[1]{\quad\text{#1}\quad}
\newcommand{\ip}[2]{\langle{#1},{#2}\rangle}

\allowdisplaybreaks

\newtheorem{theorem}{Theorem}
\newtheorem{lemma}[theorem]{Lemma}
\newtheorem{proposition}[theorem]{Proposition}
\newtheorem{corollary}[theorem]{Corollary}

\theoremstyle{remark}
\newtheorem*{remark}{Remark}

\numberwithin{equation}{section}
\numberwithin{theorem}{section}

\begin{document}

\title[Expansion and collapse of elastic bodies]{Expansion and collapse of
spherically symmetric  isotropic elastic bodies
   surrounded by vacuum\\}

 \author{Thomas C.\ Sideris}
\address{Department of Mathematics\\
University of California\\
Santa Barbara, CA 93106\\USA} 
\email{sideris@math.ucsb.edu}

\date{}



\begin{abstract}
A class of isotropic and scale invariant strain energy functions
is given for which the corresponding spherically symmetric elastic motion includes
 bodies whose diameter becomes infinite with time
or collapses to zero in finite time, depending on the sign of the  residual pressure.
The bodies are surrounded by vacuum so that
  the boundary surface forces vanish,
 while the  density remains strictly positive.  
 The body is  subject only to internal elastic stress.
 \end{abstract}

\maketitle

\section{Introduction}
We shall be concerned with $C^2$
spherically symmetric and separable motions  of 
a three-dimensional hyperelastic material based
on a class of isotropic and scale invariant strain energy
functions.  
The  solid elastic body is surrounded by vacuum so that the boundary surface 
force vanishes, while the boundary density remains strictly positive.
The body is subject only to internal elastic stress.
Depending on the sign of the residual pressure, we shall show that
the diameter of a spherical body can expand to infinity with time
or it can collapse to zero in finite time.

In addition to the assumptions of objectivity and isotropy,
we shall impose the more severe restriction of scale invariance
on the strain energy function $W$.
That is, $W$ is a homogeneous function of degree $\hd$ in the deformation
gradient $F$.  In the next section, we will show that these basic
assumptions imply that $W$ has the form
\begin{equation*}
W( F)=(\det F)^{\hd/3} W(\Sigma(F))=(\det F)^{\hd/3} \Phi(\tr \Sigma(F),\tr\cof \Sigma(F)),
\end{equation*}
for all $F\in\glt$, in which 
\begin{equation*}
\Sigma(F)=(\det F)^{-1/3}(FF^\top)^{1/2},
\end{equation*}
 is the shear strain tensor, see \cite{Shadi-1998}.
 The factor $(\det F)^{\hd/3}$ accounts for compressibility,
and the quantity $\kk(\hd)=\tfrac\hd3\left(\tfrac\hd3-1\right)$
is proportional to  the bulk modulus at $F=I$.
It is physically natural and mathematically advantageous  to assume
that $\kk(\hd)>0$, and so, we take $\hd\in\rr\setminus[0,3]$. 
The function $\Phi$ measures the resistance of the material to
shear.
The special case of a polytropic fluid arises when 
$\Phi$ is constant and
$\hd/3=-(\gamma-1)$, where
$\gamma>1$ is the adiabatic index.
Here, however, we shall focus on the case where $\Phi$ is far from a constant.
This will be measured by the size of a parameter $\beta$ which is proportional to the shear modulus.
For example, an admissible choice would be
\begin{equation}
\label{example}
W(\Sigma(F))= 
1 + c_1\left(\third\tr \Sigma(F)-1\right)+c_2\left(\third\tr\cof\Sigma(F)-1\right),
\end{equation}
with $c_1$, $c_2>0$.  This function satisfies
\begin{equation*}
W(\Sigma(F))-1\sim
 \beta|\Sigma(F)-I|^2, \quad \beta\equiv c_1+c_2>0,
\end{equation*}
 in a neighborhood of $\Sigma(F)=I$.
More generally, higher order terms of the form $ \mathcal O\left(\beta|\Sigma(F)-I|^3\right)$
may be included.
We will return to this example in Section \ref{constsec}.

Little is known
 about the long-time  behavior of solutions to the initial free boundary
 value problem in elastodynamics.
 In order to gain some insight into the possible 
 behavior, we shall investigate the  restricted class of separable motions,
 the existence of which is dependent upon the scale invariance hypothesis mentioned above.
We shall call a motion {\em separable} if its material description has the form 
\begin{equation*}
x(t,y)=\tdsf(t)\varphi(y),
\end{equation*}
in which  $\tdsf:[0,\tau)\to\rr_+$ is a scalar function and 
$\varphi:\bb\to\rr^3$ is a time-independent deformation
of the reference domain $\bb\subset \rr^3$.
Separable motions are self-similar in spatial coordinates.
In order to have nonconstant shear $\Phi$, the function $\tdsf(t)$ must be a scalar.
This contrasts with the the case of polytropic fluids 
where there exist  affine motions with $\tdsf(t)$  taking values in $\glt$.

Under the separability assumption, the spatial configuration
of the body evolves by simple dilation, whereby the scalar 
$a(t)$ controls the diameter.
The equations of motion split into
an ordinary differential equation for the scalar $a(t)$ and
an eigenvalue problem for a nonlinear partial differential
equation involving  the deformation $\varphi(y)$.
The evolution of $a(t)$ depends on the sign of the residual pressure,
which turns out to be  $-\sgn\hd$.
When $\hd<0$,  the body continuously expands for all time with
$a(t)\sim t$, as $t\to\infty$.
On the other hand, when $\hd>3$, we have
$a(t)\to0$, as $t\to\tau<\infty$, so that 
the body collapses to a point in finite time.

The main effort, then, will be devoted to solving the 
nonlinear eigenvalue problem for the deformation $\varphi$ in $C^2$.  
This will be carried out
under the assumption of spherical symmetry,
consistent with the objectivity and isotropy of $W$,
whence the PDE for $\varphi$ reduces to an ODE.  
For spherically symmetric bodies, the boundary surface force is
a pressure, and
the nonlinear vacuum (traction) boundary condition 
requires  that the 
 pressure vanishes on $\partial\bb$.
 The vacuum boundary condition
shall be fulfilled with the material in a nongaseous phase, i.e.\ with
  strictly positive  density on $\partial\bb$. 
This also contrasts with the results 
on affine compressible fluid motion 
where the vacuum boundary condition holds 
in a gaseous phase, i.e. both pressure and density vanish
on the boundary.

The existence of a  family of spherically symmetric
eigenfunctions $\{\varphi^\mu\}$ close to the identity deformation with eigenvalue
$|\mu|\ll1$  will be established in Section \ref{efunctions}
by a perturbative fixed point argument,
for every value of the elastic moduli  $\kk(\hd)>0$
and  $\beta>0$.
The behavior of $W(\Sigma(F))$ restricted to the set of spherically symmetric
deformation gradients plays a decisive role, see Section \ref{sewsi}.
If $\beta$ is sufficiently large, then there exists an eigenvalue 
for which the eigenfunction satisfies the nongaseous vacuum boundary condition.
The positivity of the shear parameter $\beta$ rules out the hydrodynamical  case.
A detailed statement of the  existence results for
expanding and collapsing spherically symmetric separable motion follows in Section \ref{mainresult}.

In the final section, we aim to persuade the reader that the assumptions imposed on the strain
energy function are physically plausible.
We show that  any self-consistent choice for the values 
of $W(\Sigma(F))$, restricted to the spherically symmetric
deformation gradients, can be extended to all deformation gradients,
and we also show that the assumptions  are consistent
with the Baker-Ericksen condition \cite{Baker-Ericksen-1954}.

\subsection*{Related literature}
The equations of motion for nonlinear elastodynamics  with a  vacuum boundary condition
are locally well-posed in Sobolev spaces under appropriate coercivity conditions,
 see for example
\cite{Stoppelli-1954}, 
 \cite{Shibata-Kikuchi-1989},
 \cite{Shibata-1989},  
 \cite{Koch-1993}.
Local well-posedness for compressible  fluids with  a
liquid boundary condition was examined in \cite{Lindblad-2005},
\cite{Coutand-Hole-Shkoller-2013}
and with a gaseous vacuum boundary condition in
 \cite{Coutand-Shkoller-2012}, \cite{Jang-Masmoudi-2015}  respectively.

Affine motion for compressible hydrodynamical models has been studied extensively,
see for example \cite{Anisimov-1970}, \cite{Dyson-1968}, \cite{holm-1981}, \cite{Ovsiannikov-1956},
but without explicit discussion of boundary conditions.
Global in time expanding affine motions for compressible 
ideal fluids satisfying the gaseous vacuum boundary conditions were constructed and
analyzed in 
\cite{Sideris-2017}. 

For   {\em self-gravitating} polytropic fluids in the mass critical case, $\gamma=4/3$ (i.e. $\hd=-1$),
spherically symmetric self-similar collapsing solutions satisfying the gaseous vacuum boundary condition
were first studied numerically in \cite{Goldreich-Weber-1980} and later constructed rigorously
in \cite{Makino-1992}, \cite{Fu-1998}, \cite{Deng-2003}.  
In the mass supercritical case, $1<\gamma<4/3$,  the existence of spherically symmetric collapsing
solutions with continuous mass absorption at the origin was established in \cite{Guo-Hadzic-Jang-2021}.

  An interesting recent article \cite{Calogero-2021}
considers the separable (the term homologous is used instead)
motion of self-gravitating elastic balls in the mass critical case $\hd=-1$.  Expanding solutions are
constructed with a solid vacuum  boundary condition, and collapsing solutions
with a gaseous vacuum boundary condition
are predicted on the basis of numerical simulations.

We emphasize that  the present work neglects self-gravitation
and  external forces. 
The sign of the  residual pressure alone determines whether the
body collapses or expands.

\section{Notation and basic assumptions}
\label{notationsec}
We denote by $\mm^3$ the set of $3\times3$ matrices over $\rr$ with the Euclidean inner product
\[
\ip{A}{B}=\tr AB^\top.
\]
We define the groups
\begin{align*}
&\glt=\{F\in\mm^3:\det F>0\}\\
&\slt=\{V\in\glt:\det V=1\}\\
&\sot=\{U\in\slt:U^{-1}=U^\top\}.
\end{align*}

Let 
\begin{subequations}
\begin{equation}
\label{sef}
W:\glt\to[0,\infty)
\end{equation}
be a smooth strain energy function.
We shall assume that $W$ is objective: 
\begin{equation}
\label{objective}
W(F)=W(UF),\isp{for all} F\in\glt,\;U\in\sot,
\end{equation}
and isotropic:
\begin{equation}
\label{isotropic}
W(F)=W(FU),\isp{for all} F\in\glt,\;U\in\sot.
\end{equation}
Conditions \eqref{sef}, 
\eqref{objective}, \eqref{isotropic} allow for spherically symmetric motion.
 Finally, we assume that $W$ 
is scale invariant, that is, it is homogeneous\footnote
{Use of the term {\em homogeneous}  here
should not be confused with
the distinct notion of a {\em homogeneous material}
 which  in continuum mechanics refers to the independence
of the strain
energy function with respect to the material coordinates  in some reference
configuration.  This has been tacitly assumed in \eqref{sef}.} of degree
$\hd$ in $F$ for some $\hd\in\rr$:
\begin{equation}
\label{homogeneity}
\wh(\scp F)=\scp^{\hd}\wh(F),\isp{for all} F\in\glt,\;\scp\in\rr_+.
\end{equation}
\end{subequations}
Homogeneity of $W$ in $F$ is necessary in order to obtain separable motions.
Since $W(I)=\sigma^{-\hd}W(\sigma I)$, $\sigma>0$, and
since we expect on physical grounds that $W(\sigma I)>0$, for $\sigma\ne1$, 
we assume that $W(I)=1$.

Using the polar decomposition, it follows from objectivity \eqref{objective} that
\begin{equation*}
W(F)=W\left((F F^\top)^{1/2}\right),\quad F\in\glt.
\end{equation*}
The positive definite symmetric matrix $A(F)=(F F^\top)^{1/2}$ is called
the  left stretch tensor,  and its  eigenvalues are the principal stretches.

With the additional assumption of homogeneity \eqref{homogeneity}, we have
\begin{equation}
\label{sefhomo1}
W(F)=(\det F)^{\hd/3}W(\Sigma(F)),\quad \isp{for all} F\in\glt,
\end{equation}
where
\begin{equation*}
\label{sefhomo2}
\Sigma(F)=(\det F)^{-1/3}(FF^\top)^{1/2}=\det A(F)^{-1/3}A(F)
\end{equation*}
is called the shear strain tensor.  Note that $\Sigma(F)\in\slt$.
The term $W(\Sigma(F))$ measures the response of the material to shear.

If  $W$ is isotropic \eqref{isotropic}, then
$W(\Sigma(F))$ must be  a function of the principal invariants of $\Sigma(F)$,
(see for example \cite{Ogden-1984}, Section 4.3.4).
Since $\Sigma(F)\in\slt$, the nontrivial invariants are
\begin{subequations}
\begin{equation}
\label{invsst}
\begin{aligned}
&H_1(\Sigma(F))=\tfrac13\tr\Sigma(F)\\
&H_2(\Sigma(F))=\tfrac13\tr\cof\Sigma(F),
\end{aligned}
\end{equation}
with the normalizing factor of $1/3$ included so that $H_i(I)=1$, $i=1,2$.
Thus, under the assumptions \eqref{sef}, \eqref{objective}, \eqref{isotropic}, \eqref{homogeneity},
the strain energy function takes the form
\begin{equation}
\label{oihsef}
W(F)=(\det F)^{\hd/3}\Phi(H_1(\Sigma(F)),H_2(\Sigma(F))),
\end{equation}
for some function
\begin{equation}
\label{oihsef1}
\Phi:\rr^2_+\to[0,\infty),\isp{with}\Phi(1,1)=W(I)=1.
\end{equation}
\end{subequations}

Note that $\det F$ and the invariants of $\Sigma(F)$ depend smoothly
on $F$ (see \cite{Gurtin-1981}, Section 3).
Therefore,  if $\Phi$ is $C^k$, then \eqref{invsst}, \eqref{oihsef}, \eqref{oihsef1} defines a $C^k$ function  $W(F)$ satisfying
\eqref{sef}, \eqref{objective}, \eqref{isotropic}, \eqref{homogeneity}.

Associated to $W$, we define its  (first) Piola-Kirchhoff stress
\begin{subequations}
\begin{equation}
\label{Piola}
S:\glt\to\mm^3,\quad S(F)=\frac{\partial W}{\partial F}(F)
\end{equation}
and Cauchy stress
\begin{equation}
\label{Cauchy}
T:\glt\to\mm^3,\quad T(F)=(\det F)^{-1}S(F)F^\top.
\end{equation}
If $W$ satisfies  \eqref{homogeneity}, then 
by differentiation  with respect to $F$
we find 
\begin{equation}
\label{Piolascaling}
S(\sigma F)=\sigma^{\hd-1}S(F),\isp{for all} F\in\glt,\;\scp\in\rr_+.
\end{equation}

\end{subequations}

\section{Equations of motion for separable solutions}
\label{sepsec}
We shall be concerned with the problem of constructing certain smooth motions of 
an elastic body whose reference configuration $\bb$ is the unit sphere 
\[
\bb=\{y\in\rr^3:|y|<1\}.
\]
  A motion is a
time-dependent family of orientation-preserving deformations $x(t,y)$
\[
x:[0,\tau)\times\bbbar\to\rr^3,
\]
with
\[
 D_yx:[0,\tau)\times\bbbar\to\glt.
\]
The image, $\Omega_t$,
of $\bb$ under the deformation $x(t,\cdot)$ represents the spatial  configuration of an
elastic body at time $t$.  The spatial description of the  body can be given in terms
of the velocity vector ${\mathbf u}(t,x)=D_tx(t,y(t,x))$ and density 
$\varrho(t,x)=\bar\varrho/\det D_yx(t,y(t,x))$ where $y(t,\cdot)=x^{-1}(t,\cdot)$ is the reference
map taking  the spatial domain $\Omega_t$ back to the material domain $\bb$ and
$\bar\varrho>0$ is the constant reference density.

The governing equations of elastodynamics, in the absence of external forces, can be written
in the form
\begin{subequations}
\begin{equation}
\label{elasticityPDE}
\br D_t^2x-D_y\cdot S(D_yx)=0,\isp{in}[0,\tau)\times\bb,
\end{equation}
subject to the nonlinear vacuum boundary condition
\begin{equation}
\label{elasticityBC}
S(D_yx(t,y))\; \omega=0,\isp{on}[0,\tau)\times\partial\bb,
\end{equation}
\end{subequations}
where $\omega=y/|y|$ is the normal at  $y\in\partial\bb$.
The initial conditions
\[
x(0,y),\quad D_tx(0,y),\quad y\in\bbbar
\]
are also  prescribed.  Local well-posedness for this system was
studied in \cite{Koch-1993}, \cite{Shibata-Kikuchi-1989},
 \cite{Shibata-1989},  \cite{Stoppelli-1954}.

\begin{remark}
In the case of  polytropic fluids, 
\[
W(F)=(\det F)^{\hd/3}=(\det F)^{-(\gamma-1)}, 
\quad \gamma>1,
\]
the Cauchy stress is
\[
T(F)=-(\det F)^{-\gamma}I.
\]
The vacuum boundary condition  can only be fulfilled with vanishing
density, i.e.  $(\det F)^{-1}=0$ on $\partial\bb$.  
In the sequel, we shall solely consider the case of 
nonvanishing density on $\partial\bb$, in order that $F\in\glt$ on $\bbbar$.
\end{remark}

We shall now impose the major restriction of separability, namely, that the motion
can be written in the form
\begin{equation}
\label{separability0}
x(t,y)=\tdsf(t)\varphi(y),
\end{equation}
for some scalar function
\[
\tdsf:[0,\tau)\to\rr_+
\]
and a time-independent orientation-preserving deformation
\[
\varphi:\bbbar\to\rr^3\isp{with}D_y\varphi:\bbbar\to\glt.
\]
Thus, the spatial configuration of an elastic body under a separable
motion evolves by dilation, $\omt=\{x=\tdsf(t)y:y\in\bbbar\}$.

In spatial coordinates,
the reference map,  velocity,  and density of a separable motion are self-similar
\begin{equation*}
\begin{aligned}
y(t,x)&=\varphi^{-1}\left(\tdsf(t)^{-1}x\right),\\
{\mathbf u}(t,x)&=\dot \tdsf(t)\varphi(y(t,x))={\dot\tdsf(t)}{\tdsf(t)^{-1}}x,\\
\varrho(t,x)&=\bar\varrho/\det[\tdsf(t) D_y\varphi(y(t,x))]\\
&=\tdsf(t)^{-3}\bar\varrho\det D_x\varphi^{-1}\left(\tdsf(t)^{-1}x\right),
\end{aligned}
\end{equation*}
for $ x\in\omt$.  We shall, however, continue to work in material coordinates.

Henceforth, we take $\br=1$.

\begin{lemma}
\label{separable}
Let $\hd, \mu\in\rr$.  Assume that $W$ satisfies \eqref{sef}, \eqref{homogeneity},
and let $S$ be defined by \eqref{Piola}.

Suppose that $\tdsf\in C^2\left([0,\tau)\right)$ is a positive solution of
\begin{equation}
\label{ODE1}
\ddot\tdsf(t) = \mu\tdsf(t)^{\hd-1},\isp{on}[0,\tau).
\end{equation}
Suppose that $\varphi\in C^2(\bb,\rr^3)\cap C^1(\bbbar,\rr^3)$ 
is an orientation-preserving deformation
which solves
\begin{subequations}
\begin{equation}
\label{elliptic1}
D_y\cdot \pksh(D_y\varphi(y))=\mu\varphi(y),\isp{in}\bb 
\end{equation}
and satisfies the boundary condition
\begin{equation}
\label{BC1}
 S(D_y\varphi(y)) \omega
=0,\isp{on}\partial\bb.
\end{equation}
\end{subequations}
Then
\[
x(t,y)=\tdsf(t)\varphi(y),\quad (t,y)\in[0,\tau)\times\bbbar
\]
is a motion  satisfying the elasticity system \eqref{elasticityPDE}, \eqref{elasticityBC},
with $\br=1$.

In addtion, $\varphi$ satisfies
\[
\mu\int_\bb|\varphi(y)|^2dy
=-\hd\int_\bb W(D\varphi(y))dy,
\]
and $-\hd\mu\ge0$.
\end{lemma}

\begin{proof}
Since $\tdsf$  is assumed to be  positive
and $\varphi$ is assumed to be a deformation, $x(t,y)$, as defined, is a motion.

By
\eqref{Piolascaling}, \eqref{ODE1}, \eqref{elliptic1}, the motion $x$ satisfies the 
 system \eqref{elasticityPDE}:
 \begin{multline*}
D_t^2x(t,y)
=\ddot \tdsf(t)\varphi(y)\\
=\mu \tdsf(t)^{\hd-1} \varphi(y)
=\tdsf(t)^{\hd-1}D_y\cdot\pksh(D_y\varphi(y))\\
=D_y\cdot S(\tdsf(t)D_y\varphi(y))
=D_y\cdot S(D_yx(t,y)).
\end{multline*}

The boundary condition \eqref{elasticityBC} is similarly verified using 
\eqref{BC1} and the homogeneity of $S$ in $F$:
\begin{equation*}
S(D_yx(t,y))\omega=S(\tdsf(t)D_y\varphi(y))\omega=\tdsf(t)^{\hd-1}S(D_y\varphi(y))\omega=0.
\end{equation*}

Finally, by \eqref{homogeneity},
\[
\hd W(F)=\tfrac{d}{d\sigma}\sigma^\hd W(F)\big|_{\sigma=1}
=\tfrac{d}{d\sigma}W(\sigma F)\big|_{\sigma=1}=\ip{S(F)}{F},
\]
for all $F\in\glt$.
So, any solution of \eqref{elliptic1}, \eqref{BC1} with $D\varphi\in\glt$ satisfies
\[
\begin{aligned}
\mu\int_\bb|\varphi(y)|^2dy
=&\int_\bb \ip{D\cdot S(D\varphi(y))}{\varphi(y)}dy\\
=&-\int_\bb\ip{S(D\varphi(y))}{D\varphi(y)}dy\\
=&-\hd\int_\bb W(D\varphi(y))dy.
\end{aligned}
\]
\end{proof}

\begin{remark}
The PDE \eqref{elliptic1} is the Euler-Lagrange equation associated to the action
\begin{equation}
\label{action}
\int_\bb\left(W(D_y\varphi)+\tfrac\mu2|\varphi|^2\right)dy.
\end{equation}
\end{remark}

We shall consider the initial value problem for \eqref{ODE1} in the next section.
Sections \ref{ssdsec}--\ref{efunctions} will be devoted to the solution of
eigenvalue problem \eqref{elliptic1},\eqref{BC1}.

\section{Dynamical behavior}

\begin{lemma}
\label{dyn}
If $\tdsf(t)$ is a $C^2$ positive solution of  \eqref{ODE1} with $\hd\ne0$, then the  quantity 
\begin{equation}
\label{scalarenergy}
E(t)=\half \dot\tdsf(t)^2-\tfrac{\mu}{\hd}\tdsf(t)^{\hd}
\end{equation}
is conserved.  

If  $\hd<0$ and   $\mu>0$, 
then for every $(\tdsf(0),\dot \tdsf(0))\in\rr_+\times\rr$,  
the initial value problem for \eqref{ODE1} 
has a positive solution $\tdsf\in C^2\left([0,\infty)\right)$
with $0< (2E(0))^{1/2}-a(t)/t\to0$, as $t\to\infty$.

If  $\hd>1$ and $\mu<0$,   
then for every $(\tdsf(0),\dot \tdsf(0))\in\rr_+\times\rr$,  
the initial value problem for \eqref{ODE1} 
has a positive solution $\tdsf\in C^2\left([0,\tau)\right)$,
with $\tau<\infty$
and
$\tdsf(t)\to0$, as $t\to\tau$.
\end{lemma}

\begin{proof}
If $(\tdsf(0),\dot \tdsf(0))\in\rr_+\times\rr$, then 
the initial value problem for \eqref{ODE1} has a $C^2$ positive solution on a
maximal interval $[0,\tau)$.  If $\tau<\infty$, then either $\tdsf(t)\to0$ or $\tdsf(t)\to\infty$,
as $t\to\tau$.

Conservation of $E(t)$ on the interval $[0,\tau)$ follows directly from \eqref{ODE1}.

Assume that $\hd<0$ and $\mu>0$.  Then by \eqref{scalarenergy},  $\tdsf(t)^{-1}$  and $|\dot\tdsf(t)|$ are bounded above by
some constant $C_0$ on $[0,\tau)$.  This implies that $C_0^{-1}\le \tdsf(t)\le\tdsf(0)+C_0t$,
on $[0,\tau)$.  It follows that  $\tau=+\infty$.

Let $X(t)=\half\tdsf(t)^2$.  Then
\begin{equation*}
\ddot X(t)=\dot\tdsf(t)^2+\mu\tdsf(t)^{\hd}.
\end{equation*}
From \eqref{scalarenergy}, we obtain
\begin{subequations}
\begin{equation}
m E(0)\le \ddot X(t)\le M E(0),
\end{equation}
in which 
\[
m=\min\{2,-\hd\}\isp{and} M=\max\{2,-\hd\}.
\]
This leads to the bounds
\begin{equation}
\label{xdotbd}
m E(0)t\le \dot X(t)-\dot X(0)\le ME(0)t
\end{equation}
and
\begin{equation}
\label{xbd}
m E(0)t^2\le X(t)-\dot X(0)t-X(0)\le ME(0)t^2.
\end{equation}
\end{subequations}

By \eqref{xdotbd}, we see that 
\begin{equation*}
\dot a(t)=\dot X(t)/\tdsf(t)>0,\isp{for}t>-\dot X(0)/(mE(0)),
\end{equation*}
and by
\eqref{xbd}, we deduce that 
\begin{equation*}
\tdsf(t)=(2X(t))^{1/2}\sim t, \isp{as} t\to\infty.
\end{equation*}
With these facts, $E(t)=E(0)$ implies that 
\begin{equation*}
0< (2E(0))^{1/2}-\dot\tdsf(t)\to0, \isp{as} t\to\infty,
\end{equation*}
from which follows the asymptotic statement.

Now assume that $\hd>1$ and  $\mu<0$.
Let $\underaccent{\bar}{\tdsf}=\inf\{a(t):t\in[0,\tau)\}$.
Since $\mu<0$, \eqref{ODE1} implies that
\[
\ddot a(t)=\mu a(t)^{\hd-1}(t)\le \mu \underaccent{\bar}{\tdsf}^{\hd-1}
\isp{ on} [0,\tau),
\]
and  so  
we obtain
\begin{equation}
\label{tdsfub}
0\le \tdsf(t)\le \tdsf(t_0)+\dot\tdsf(t_0)(t-t_0)+\half\mu\underaccent{\bar}{\tdsf}^{\hd-1} (t-t_0)^2,
\end{equation}
for all $0\le t_0\le t<\tau$.
If $\underaccent{\bar}{\tdsf}>0$,
then, since $\mu<0$, it follows from \eqref{tdsfub}
that $\tau<\infty$.  On the other hand,
if $\underaccent{\bar}{\tdsf}=0$, then since $a(0)>0$, 
there exists a $t_0\in[0,\tau)$ such that $\dot\tdsf(t_0)<0$, whence from \eqref{tdsfub}
again there holds $\tau<\infty$.  
\end{proof}

\begin{remark}
When $\hd>1$ and $\mu<0$, the time of collapse is given by
\begin{equation*}
\tau=
\begin{cases}
\displaystyle
\int_0^{\tdsf(0)}\left[2\left(E(0)+\tfrac{\mu}{\hd}s^\hd\right)\right]^{-1/2}ds,
&\text{if $\dot\tdsf(0)\le0$}\\
\ \\
\displaystyle
\left(\int_0^{a\left(\dot a^{-1}(0)\right)}+\int_{\tdsf(0)}^{a\left(\dot a^{-1}(0)\right)}\right)
\left[2\left(E(0)+\tfrac{\mu}{\hd}s^\hd\right)\right]^{-1/2}ds,
&\text{if $\dot\tdsf(0)>0$}.
\end{cases}
\end{equation*}

\end{remark}

\begin{remark}
In Lemma \ref{dyn}, we have only  discussed the qualitative behavior of solutions to
equation \eqref{ODE1} for the parameter range $-\hd\mu>0$ from Lemma \ref{separable}
in which we can construct separable solutions
to \eqref{elasticityPDE}, \eqref{elasticityBC}.
\end{remark}

\section{Spherically symmetric deformations}
\label{ssdsec}

\begin{lemma}
\label{ssdgform}
If 
\begin{equation}
\label{phicond}
\phi\in C^2\left([0,1]\right), \quad\phi(0)=\phi''(0)=0,  \isp{and} \phi'>0 \isp{on} [0,1], 
\end{equation}
then
\[
\varphi(y)=\phi(r)\omega,\quad r=|y|,\quad \omega=y/|y|
\]
defines an orientation-preserving deformation $\varphi\in C^2(\bbbar,\rr^3)$.

The function
\begin{equation*}
\label{lambpos}
\lambda(r)=(\lambda_1(r),\lambda_2(r))=(\phi'(r),\phi(r)/r)
\end{equation*}
belongs to $C^1([0,1],\rr_+^2)$, positivity holds: $\lambda_1(r),\lambda_2(r)>0$ on $[0,1]$, and 
$\lambda(0)=(\phi'(0),\phi'(0))$.

The deformation gradient of $\varphi$, $D_y\varphi:\bbbar\to\glt$,
  is given by
\begin{subequations}
\begin{equation}
\label{sph-sym-def-gr}
D_y\varphi(y)
=F(\lambda(r),\omega),\quad r=|y|,\quad \omega=y/|y|,
\end{equation}
with
\begin{equation}
\label{ssdg}
F(\lambda,\omega)=\lambda_1P_1(\omega)+\lambda_2P_2(\omega),\quad\lambda\in\rr_+^2,\quad \omega\in\ S^2,
\end{equation}
and
\begin{equation}
\label{projections}
P_1(\omega)=\omega\otimes\omega,\quad  P_2(\omega)=I-P_1(\omega).
\end{equation}
\end{subequations}
\end{lemma}

\begin{proof}
By writing
\[
\phi(r)=\int_0^1D_s\left(\phi(sr)\right)ds=r\int_0^1\phi'(sr)ds,
\]
we see that the function 
\[
\lambda_2(r)=\phi(r)/r=\int_0^1\phi'(sr)ds
\]
 is strictly positive and $C^1$ on $(0,1]$
with 
\[
\lambda_2(r)\to\phi'(0)>0,\quad
 \lambda_2'(r)=\int_0^1s\phi''(sr)ds\to\half\phi''(0)=0,\isp{as}r\to0.
\]
Thus, $\lambda_2(r)$ extends to a strictly positive   function
in $C^1([0,1])$.  It follows that the function
\[
\chi(r)=\phi'(r)-\phi(r)/r=\lambda_1(r)-\lambda_2(r)
\]
belongs to $C^1([0,1])$, and $\chi(0)=\chi'(0)=0$.

Clearly, $\varphi\in C^2(\bbbar\setminus\{0\})$, so to prove that $\varphi\in
C^2(\bbbar)$, we need only show that the derivatives up to second order
extend continuously to the origin.

The formula \eqref{sph-sym-def-gr} for $D_y\varphi(y)$ is easily verified for $y\ne0$,
and it equivalent to
\begin{subequations}
\begin{equation}
\label{d1phi}
D_y\varphi(y)
=\chi(r)P_1(\omega)+\lambda_2(r)I.
\end{equation}
This implies  that
\[
D_y\varphi(y)\to\phi'(0)I,\isp{as} y\to0,
\]
and so $\varphi\in C^1(\bbbar)$.

From the formula \eqref{sph-sym-def-gr}, we also see that $D_y\varphi(y)$ has the eigenspaces
$\spn\{\omega\}$ and $\spn\{\omega\}^\perp$, with strictly positive eigenvalues
$\lambda_1(r)$, $\lambda_2(r)$, $\lambda_2(r)$.  Thus, 
\[
\det D_y\varphi(y)=\lambda_1(r)\lambda_2(r)^2, 
\isp{for all} y\in\bbbar,
\]
which shows that $D_y\varphi:\bbbar\to\glt$.  Since $\phi'>0$, we see that $\varphi$ is
a bijection of $\bbbar$ onto its range.  So we conclude that $\varphi$ is a $C^1$ 
orientation-preserving  deformation on $\bbbar$.

It remains to show that $D_y\varphi\in C^1(\bbbar)$.  
For $y\ne0$, we find from \eqref{d1phi} that
\begin{multline}
\label{d2phi}
D_jD_k\varphi_i(y)=(\chi'(r)-2\chi(r)/r)\omega_i\omega_j\omega_k\\
+(\chi(r)/r)\left(\delta_{ij}\omega_k+\delta_{ik}\omega_j+\delta_{jk}\omega_i\right).
\end{multline}
\end{subequations}
Since $\chi(0)=\chi'(0)=0$, we have that $\chi'(r), \;\chi(r)/r\to\chi'(0)=0$,
as $r\to 0$.
Thus,
we see that the second derivatives of $\varphi$  extend to  continuous functions on  $\bbbar$
which  vanish at the origin.
\end{proof}

A deformation of the form $\varphi(y)=\phi(r)\omega$
is said to be {\em spherically symmetric}.
From now on, we focus exclusively upon $C^2$ spherically symmetric
orientation-preserving deformations
of the reference domain $\bbbar$ where $\phi$ satisfies \eqref{phicond}.  

\begin{remark}
The spatial configuration of a body at time $t$ under
a spherically symmetric separable motion 
$x(t,y)=\tdsf(t)\phi(r)\omega$, defined on $[0,\tau)\times\bbbar$,  
is a sphere of
radius $a(t)\phi(1)$.
\end{remark}

The next  result summarizes the properties
 of the gradient of a spherically symmetric deformation.

\begin{lemma}
\label{flomdef}
The matrix $F(\lambda,\omega)$ defined in \eqref{ssdg}, \eqref{projections}
satisfies
\begin{itemize}
\item
the eigenspaces of $F(\lambda,\omega)$ are $\spn\{\omega\}$ and $\spn\{\omega\}^\perp$,
\item
the eigenvalues of $F(\lambda,\omega)$ are $\lambda_1,\lambda_2,\lambda_2$,
\item
$\det F(\lambda,\omega)=\lambda_1\lambda_2^2$,
\item
$F:\rr_+^2\times S^2\to\glt$,
\item
$F(\lambda,\omega)$ is positive definite  symmetric, and
\item
$F(\lambda,\omega)=\left(F(\lambda,\omega)F(\lambda,\omega)^\top\right)^{1/2}=A(F(\lambda,\omega))$.
\end{itemize}
\end{lemma}

\section{Spherically symmetric strain energy and stress}

\begin{lemma}
\label{sph-sym-str-en}
Let $W$ be a strain energy function satisfying \eqref{sef}, \eqref{objective}, \eqref{isotropic}.
If $F(\lambda,\omega)$ 
 is given by \eqref{ssdg}, \eqref{projections},
then $W\circ F(\lambda,\omega)$ is independent of $\omega$.
\end{lemma}

\begin{proof}
Fix a vector $\omega_0\in S^2$.  For an arbitrary vector $\omega\in S^2$, 
choose  $U\in\sot$ such that $\omega=U\omega_0 $.
Then
\[
F(\lambda,\omega)=F(\lambda,U\omega_0 )=\lambda_1 P_1(U\omega_0 )
+\lambda_2P_2(U\omega_0 )=UF(\lambda, \omega_0 )U^\top.
\]
From \eqref{objective} and \eqref{isotropic}, we obtain
\[
W\circ F(\lambda,\omega)=W\circ F(\lambda,\omega_0 ),
\]
which is independent of $\omega$.
\end{proof}
Using the result of Lemma \ref{sph-sym-str-en}, we may define
a $C^2$ function $\llll$ by 
\begin{subequations}
\begin{equation}
\label{ssse}
\llll:\rr_+^2\to[0,\infty),\quad 
\llll(\lambda)=W\circ F(\lambda,\omega).
\end{equation}
In other words, $\llll$ is the restriction of $W$ to the set of spherically symmetric
deformation gradients.  By  \eqref{ssdg}, \eqref{homogeneity}, $\llll$ scales like $W$:
\begin{equation}
\label{ellscaling}
\llll(\scp\lambda)=W(F(\scp\lambda,\omega))=W(\scp F(\lambda,\omega))
=\scp^\hd\llll(\lambda),\quad \scp\in\rr_+.
\end{equation}
\end{subequations}

We now obtain expressions for the 
stresses  
restricted to the set of spherically symmetric deformation gradients.
\begin{lemma}
\label{P-K-str}
If $W$ satisfies \eqref{sef}, \eqref{objective}, \eqref{isotropic},
and $\llll$ is defined by \eqref{ssse},
then the Piola-Kirchhoff stress defined in \eqref{Piola} satisfies
\begin{subequations}
\begin{equation}
\label{sspksell}
S\circ F(\lambda,\omega)=\llll_{,1}(\lambda)P_1(\omega)+\half\llll_{,2}(\lambda)P_2(\omega),
\end{equation}
and the Cauchy stress defined in \eqref{Cauchy} satisfies
\begin{equation}
\label{sscs0}
T\circ F(\lambda,\omega)=(\lambda_1\lambda_2^2)^{-1}\left[\lambda_1\llll_{,1}(\lambda)P_1(\omega)
+\half\lambda_2\llll_{,2}(\lambda)P_2(\omega)\right].
\end{equation}
\end{subequations}
\end{lemma}

\begin{proof}
Differentiation of \eqref{ssse} with respect to $\lambda$ yields
\begin{equation}
\label{LagrangianDerivatives}
\begin{aligned}
&\llll_{,1}(\lambda)=\ip{S\circ F(\lambda,\omega)}{P_1(\omega)}\\
&\llll_{,2}(\lambda)=\ip{S\circ F(\lambda,\omega)}{P_2(\omega)},
\end{aligned}
\end{equation}
where $\langle\cdot,\cdot\rangle$ is the Euclidean product on $\mm^3$.

It is a standard fact (\cite{Ogden-1984}, Theorem 4.2.5)
that the Cauchy stress $T(F)$  associated to an objective
and isotropic 
strain energy function satisfies 
\[
T(F)\in\spn\{I,A(F),A(F)^2\}.
\]
  By Lemma \ref{flomdef}, we have
  \[
A(F(\lambda,\omega))
=F(\lambda,\omega).
\]
Since
\[
F(\lambda_1,\lambda_2,\omega)^k=F(\lambda_1^k,\lambda_2^k,\omega)
\in\spn\{P_1(\omega),P_2(\omega)\},\quad k\in\mathbb Z,
\]
we obtain
\[
T\circ F(\lambda,\omega)\in\spn\{P_1(\omega),P_2(\omega)\}.
\]
By \eqref{Cauchy},  
there also holds
\[
S\circ F(\lambda,\omega)\in\spn\{P_1(\omega),P_2(\omega)\}.
\]
Hence, we may write
\[
S\circ F(\lambda,\omega)=c_1(\lambda,\omega)P_1(\omega)+c_2(\lambda,\omega)P_2(\omega).
\]
Taking the $\mm^3$-inner product with $P_1(\omega)$ and $P_2(\omega)$,
we have from \eqref{LagrangianDerivatives}
\[
\llll_{,1}(\lambda)=c_1(\lambda,\omega)\isp{and} \llll_{,2}(\lambda)=2c_2(\lambda,\omega).
\]
 This proves \eqref{sspksell}, and
\eqref{sscs0} now follows from \eqref{Cauchy} and Lemma \ref{flomdef}.

\end{proof}

\begin{corollary}
Under the assumptions of Lemma \ref{P-K-str}, there holds
\begin{equation}
\label{residualpressure}
\llll_{,1}(\alpha ,\alpha)=\half\llll_{,2}(\alpha ,\alpha), \isp{for all} \alpha>0.
\end{equation}

\end{corollary}

\begin{proof}
For any $\alpha>0$, the map $\varphi(y)=\alpha y=\alpha r\omega$ is a smooth spherically
symmetric deformation, and $D_y\varphi(y)=\alpha I$.  Since $S:\glt\to\mm^3$ is $C^1$,
the map $S(\alpha I)$ is $C^1$ in $\alpha$. By \eqref{sspksell}, we have
\begin{align*}
S(\alpha I)&=\llll_{,1}(\alpha ,\alpha)P_1(\omega)+\half\llll_{,2}(\alpha ,\alpha)P_2(\omega)\\
&=\left(\llll_{,1}(\alpha ,\alpha)-\half\llll_{,2}(\alpha ,\alpha)\right)P_1(\omega)+
\half\llll_{,2}(\alpha ,\alpha)I
\end{align*}
 Now $P_1(\omega)$ is bounded and discontinuous at the origin, so \eqref{residualpressure} 
must hold.
\end{proof}

\begin{corollary}
\label{residualstress}
Under the assumptions of Lemma \ref{P-K-str}, there holds
\begin{equation*}
T(\alpha I)=\alpha^{-2}\llll_{,1}(\alpha,\alpha)I\equiv-\mathcal P(\alpha) I.
\end{equation*}
\end{corollary}

\begin{proof}
This follows directly from Lemma \ref{P-K-str} and Corollary \ref{residualstress}.
\end{proof}

We define the residual stress to be $S(I)=T(I)=-\mathcal P(1)I$.
We shall refer to $\mathcal P(1)$
 as  the {\em residual pressure}.

\begin{corollary}
\label{Div-Piola}
Under the assumptions and notation of Lemmas \ref{ssdgform} and  \ref{P-K-str},
we have for any spherically symmetric orientation-preserving deformation $\varphi$
\begin{equation}
\label{sspks}
S(D_y\varphi(y))=\llll_{,1}(\lambda(r))P_1(\omega)+\half \llll_{,2}(\lambda(r))P_2(\omega)
\end{equation}
and
\begin{multline*}
T(D_y\varphi(y))
=\left(\lambda_1(r)\lambda_2(r)^2\right)^{-1}\\
\times\left(\lambda_1(r) \llll_{,1}(\lambda(r))P_1(\omega)
+\half\lambda_2(r) \llll_{,2}(\lambda(r))P_2(\omega)\right).
\end{multline*}
\end{corollary}

\section{The nonlinear eigenvalue problem}
\begin{lemma}
\label{evp2}
Suppose that  $W$ satisfies \eqref{sef}, \eqref{objective}, \eqref{isotropic} and
that $\llll$ is defined by \eqref{ssse}.

Let $\mu\in\rr$.
If $\phi$ satisfies \eqref{phicond} and
\begin{subequations}
\begin{multline}
\label{ssODE1}
D_r[\llll_{,1}(\phi'(r),\phi(r)/r)]\\+\tfrac2r[\llll_{,1}(\phi'(r),\phi(r)/r)-\half\llll_{,2}(\phi'(r),\phi(r)/r)]\\
=\mu\phi(r),
\quad r\in[0,1)
\end{multline}
then $\varphi(y)=\phi(r)\omega$ 
is a $C^2$ spherically symmetric orientation-preserving
 deformation on $\bbbar$ which solves \eqref{elliptic1}.

If 
\begin{equation}
\label{ssBC1}
\left.\llll_{,1}(\phi'(r),\phi(r)/r)\right|_{r=1}=0,
\end{equation}
\end{subequations}
then $\varphi$ satisfies the boundary condition \eqref{BC1}.
\end{lemma}

\begin{proof}
Set
\[
\Lambda_k(r)=\llll_{,k}(\phi'(r),\phi(r)/r),\quad k=1,2,
\]
so that by \eqref{sspks},
\[
S(D_y\varphi(y))
=\left(\Lambda_1(r)-\half \Lambda_2(r)\right)P_1(\omega)+\half \Lambda_2(r)I.
\]

Then, using the facts  $D_r=\omega_jD_j$, $D_j\omega_j=2/r$, and $D_jr=\omega_j$, we have

\begin{align*}
[D_y\cdot & S(D_y\varphi(y))]_i\\
={}&\left[D_y\cdot \left(\left(\Lambda_1(r)-\half \Lambda_2(r)\right)P_1(\omega)+\half \Lambda_2(r)I\right)\right]_i\\
={}&D_j\Big(\left(\Lambda_1(r)-\half \Lambda_2(r)\right)\omega_i\omega_j+\half \Lambda_2(r)\delta_{ij}\Big)\\
={}&D_r\left(\Lambda_1(r)-\half \Lambda_2(r)\right)\omega_i
+\left(\Lambda_1(r)-\half \Lambda_2(r)\right)\omega_iD_j\omega_j+\half D_i\Lambda_2(r)\\
={}&\Big( \Lambda_1'(r)+\tfrac2r( \Lambda_1(r)-\half \Lambda_2(r))\Big)\omega_i.
\end{align*}
Thus, \eqref{ssODE1} implies that \eqref{elliptic1} holds.

If $y\in\partial\bb$, then $r=1$, so we have from \eqref{sspksell}
\[
\left.S(D_y\varphi(y))\;\omega\right|_{y\in\partial\bb}=\Lambda_1(1)\omega.
\]
So \eqref{ssBC1} implies \eqref{BC1}.
\end{proof}

\begin{remark}
The ODE \eqref{ssODE1} is the Euler-Lagrange equation associated to the action
\[
4\pi\int_0^1\left[\llll\left(\phi'(r),\phi(r)/r\right)+\tfrac\mu2\phi(r)^2\right]r^2dr,
\]
which is the action \eqref{action} restricted to the set of spherically symmetric deformations.
\end{remark}

\begin{remark}
If $\llll^0:\rr_+^2\to[0,\infty)$ is $C^2$ and
\begin{equation}
\label{nullLag}
\llll^0_{,11}(\lambda)=0,\quad
(\lambda_1-\lambda_2)\llll^0_{,12}(\lambda)
+2\llll^0_{,1}(\lambda)-\llll^0_{,2}(\lambda)=0,\quad \lambda\in\rr_+^2,
\end{equation}
then \eqref{ssODE1} is unchanged by replacing $\llll$ with $\llll+\llll^0$.
Condition \eqref{nullLag} holds provided there exists a 
pair of $C^2$ functions $\nullf_i:\rr_+\to\rr$, $i=0,1$, such that
\begin{gather*}
\llll^0(\lambda)=\nullf_1(\lambda_2)\lambda_1+\nullf_0(\lambda_2)\\
\intertext{and}
 -\lambda_2\nullf_1'(\lambda_2)+2\nullf_1(\lambda_2)-\nullf_0'(\lambda_2)=0,
\end{gather*}
for all $\lambda\in\rr_+^2$.
If $W^0$ satisfies \eqref{sef}, \eqref{objective}, \eqref{isotropic} and if
$W^0$ is a null Lagrangian, see  \cite{Ball-1981}, then
$\llll^0=W^0\circ F(\lambda,\omega)$ satisfies \eqref{nullLag}.
Thus, we can think of  solutions of \eqref{nullLag} as  being spherically symmetric null Lagrangians.

\end{remark}

\section{Strain energy with scaling invariance}
\label{sewsi}

It is convenient to define the quantities 
\begin{subequations}
\begin{equation}
\label{lamudef}
\volume=\det F(\lambda,\omega)=\lambda_1\lambda_2^2, \quad
u=\lambda_1/\lambda_2,\quad \lambda\in\rr_+^2,
\end{equation}
so that
\begin{equation}
\label{lamuinv}
\lambda_1=\volume^{1/3}u^{2/3},\quad \lambda_2=\volume^{1/3}u^{-1/3},
\quad (u,v)\in\rr_+^2.
\end{equation}
\end{subequations}
Note that $u=1$ if and only if $\lambda_1=\lambda_2$ if and only if
 $F(\lambda,\omega)$ is a multiple of the
identity if and only if $\Sigma(F(\lambda,\omega))=I$.

\begin{lemma}
\label{constlem}
If $W$ is a $C^k$ strain energy function 
satisfying \eqref{sef}, \eqref{objective}, \eqref{isotropic}, \eqref{homogeneity},
and $\llll$ is defined by \eqref{ssse},
then there exists a $C^k$ function 
\[
f:\rr_+\to[0,\infty),
\isp{with}
f'(1)=0,
\]
such that 
\begin{equation}
\label{sssefrep}
\llll(\lambda)
=\volume^{\hd/3}f(u),\isp{for all}\lambda\in\rr_+^2.
\end{equation}
  \end{lemma}

\begin{proof}
Define
\begin{equation*}
f(u)=\llll(u^{2/3},u^{-1/3}),\quad u\in\rr_+.
\end{equation*}
Then $f$ is $C^k$ and nonnegative, by \eqref{ssse}.   
It follows from  Corollary \ref{residualpressure} that 
\begin{equation*}
f'(1)=\tfrac23\llll_{,1}(1,1)-\third\llll_{,2}(1,1)=0.
\end{equation*}
By \eqref{lamuinv}  and \eqref{ellscaling}, we have
\begin{equation*}
\llll(\lambda)=\llll(v^{1/3}u^{2/3},v^{1/3}u^{-1/3})=v^{\hd/3}\llll(u^{2/3},u^{-1/3})=v^{\hd/3}f(u),
\end{equation*}
so that \eqref{sssefrep} holds.
\end{proof}

A partial converse to Lemma \ref{constlem} will be given in
Theorem \ref{thm2}.
We shall also show, in Proposition \ref{BEineq}, that
condition \eqref{sssefrep} is physically plausible, insofar as it is consistent with 
the Baker-Ericksen inequality, when $f$ is convex.

\begin{remark}
We have by \eqref{sefhomo1}
\begin{equation}
\label{sefrestrss}
W(\Sigma(F(\lambda,\omega)))= \det F(\lambda,\omega)^{-\hd/3}W(F(\lambda,\omega))=f(u).
\end{equation}
That is, $f$ is the restriction of $W(\Sigma(F))$ to the spherically symmetric
deformation gradients.
\end{remark}

\begin{remark}
In the polytropic fluid case, $W(F)=(\det F)^{-(\gamma-1)}$, we see that
 $f(u)=1$.
\end{remark}

Next, we explore the implications  of \eqref{sssefrep} for the equation \eqref{ssODE1}.

\begin{lemma}
\label{evp3}
Suppose that  $W$ is a $C^3$ strain energy function which
satisfies \eqref{sef}, \eqref{objective}, \eqref{isotropic}, \eqref{homogeneity} and
that $\llll$ is defined by \eqref{ssse}.  By Lemma \ref{constlem}, there is a $C^3$ function 
\[
f:\rr_+\to[0,\infty),
\isp{with}
f'(1)=0,
\]
such that 
\begin{equation*}
\llll(\lambda)
=\volume^{\hd/3}f(u),\isp{for all}\lambda\in\rr_+^2.
\end{equation*}
Define the $C^1$ functions
\begin{subequations}
\begin{equation}
\label{biguforms}
\begin{aligned}
&U_1(u)=\kappa(\hd)f(u)+\tfrac{2\hd}3uf'(u)+u^2f''(u),\\
&U_2(u)=2\kk(\hd)uf(u)+(\tfrac\hd3-1)u^2f'(u)\\
&\phantom{U_2(u)=\tfrac{2\hd}3(\tfrac\hd3-}
+(2u^2+u^3)\frac{f'(u)}{u-1}-u^3f''(u),
\end{aligned}
\end{equation}
in which 
\begin{equation}
\label{bmdef}
\kk(\hd)=\tfrac{\hd}3\left(\tfrac\hd3-1\right).
\end{equation}

Suppose that $\phi$ satisfies \eqref{phicond}, and define the positive functions
\begin{equation}
\label{ellofr}
\lambda(r)=(\lambda_1(r),\lambda_2(r))=(\phi'(r),\phi(r)/r),
\end{equation}
and
\begin{equation}
\label{uvofr}
\volume(r)=\lambda_1(r)\lambda_2(r)^2,\quad u(r)=\lambda_1(r)/\lambda_2(r).
\end{equation}
If $\phi$ solves the equation
\end{subequations}
\begin{subequations}
\begin{multline}
\label{ODE2.0}
U_1(u(r))\phi''(r)+
U_2(u(r))
\frac1r\left(\phi'(r)-\phi(r)/r\right)\\
=\mu \; r \; \volume(r)^{-\hd/3+1} u(r),\quad r\in[0,1)
\end{multline}
and the boundary condition
\begin{equation}
\label{sshBC0}
g(u(1))=0,\isp{with}  g(u)\equiv\tfrac\hd3f(u)+uf'(u),
\end{equation}
\end{subequations}
then it also solves   \eqref{ssODE1},  \eqref{ssBC1}
and $\varphi(y)=\phi(r)\omega$ is a $C^2$ spherically symmetric 
orientation-preserving deformation
which solves \eqref{elliptic1}, \eqref{BC1}.

\end{lemma}

\begin{proof}
Carrying out the differentiation in \eqref{ssODE1}, we may write
\begin{multline*}
\llll_{,11}(\lambda(r))\;\phi''(r)+\llll_{,12}(\lambda(r))\;\tfrac1r\left(\phi'(r)-\phi(r)/r\right)\\
+\tfrac2r[\llll_{,1}(\lambda(r))-\half\llll_{,2}(\lambda(r))]=\mu \phi(r).
\end{multline*}
Since
\[
\phi(r)=rv(r)^{1/3}u(r)^{-1/3}\isp{and} \phi'(r)-\phi(r)/r=v(r)^{1/3}u(r)^{-1/3}(u(r)-1),
\]
this is equivalent to 
\begin{multline}
\label{ODE1.2}
\llll_{,11}(\lambda(r))\;\phi''(r)+\llll_{,12}(\lambda(r))\;\tfrac1r\left(\phi'(r)-\phi(r)/r\right)\\
+ v(r)^{-1/3}u(r)^{1/3}\left[\frac{2\llll_{,1}(\lambda(r))-\llll_{,2}(\lambda(r))}{u(r)-1}\right]
\tfrac1r\left(\phi'(r)-\phi(r)/r\right)\\
=\mu\; r \;v(r)^{1/3}u(r)^{-1/3}.
\end{multline}
From \eqref{lamudef}, we derive
\begin{align*}
&\partial_{\lambda_1}=\volume_{,1}\partial_\volume+u_{,1}\partial_u
=\volume^{-1/3}u^{-2/3}(\volume\partial_\volume+u\partial_u)\\
&\partial_{\lambda_2}=\volume_{,2}\partial_\volume+u_{,2}\partial_u
=\volume^{-1/3}u^{1/3}(2\volume\partial_\volume-u\partial_u).
\end{align*}

Direct computation from \eqref{sssefrep} yields
\begin{equation}
\label{ellforms}
\begin{aligned}
&\llll_{,1}(\lambda)=\volume^{(\hd-1)/3}u^{-2/3}[\tfrac\hd3f(u)+uf'(u)]\\
&\llll_{,2}(\lambda)=\volume^{(\hd-1)/3}u^{-2/3}[\tfrac{2\hd}3uf(u)-u^2f'(u)]\\
&\llll_{,11}(\lambda)\\
&\quad
=\volume^{(\hd-2)/3}u^{-4/3}\left[\kappa(\hd)f(u)+\tfrac{2\hd}3uf'(u)
+u^2f''(u)\right]\\
&\llll_{,12}(\lambda)\\
&\quad=\volume^{(\hd-2)/3}u^{-1/3}[2(\tfrac\hd3)^2f(u)+(\tfrac\hd3-1) uf'(u)
-u^2f''(u)].
\end{aligned}
\end{equation}
Upon substitution of \eqref{ellforms} relations into \eqref{ODE1.2},
we obtain \eqref{ODE2.0} 
after a bit of simplification.

By \eqref{ellforms}, the boundary condition \eqref{ssBC1} is equivalent to
\begin{equation*}
\volume(1)^{(\hd-1)/3}u(1)^{-2/3}\left[\tfrac\hd3f(u(1))+u(1)f'(u(1))\right]=0,
\end{equation*}
which reduces to \eqref{sshBC0}.

This shows  that the problems        
\eqref{ODE2.0},  \eqref{sshBC0} 
and  \eqref{ssODE1},  \eqref{ssBC1} are equivalent.  By Lemma \ref{evp2}, $\varphi(y)=\phi(r)\omega$
is a $C^2$ spherically symmetric 
orientation-preserving deformation
which solves \eqref{elliptic1}, \eqref{BC1}.
\end{proof}

\begin{remark}
The quantity $\kk(\hd)$ defined in \eqref{bmdef}
corresponds to the bulk modulus, as we shall explain in Lemma \ref{moduli}.
Materials with a negative bulk modulus are uncommon,
and therefore  it is reasonable physically to assume $\kk(\hd)>0$.  In the next lemma, 
we will also see that positivity of $\kk(\hd)$
relates to the coercivity of the differential operator in \eqref{ODE2.0}, and hence
the hyperbolicity of the  equations of motion.
\end{remark}

\begin{remark}
When \eqref{sssefrep} holds, we have
\begin{equation}
\label{respress}
-\mathcal P(\alpha)=\alpha^{-2}\llll_{,1}(\alpha,\alpha)=(\hd/3)f(1)\alpha^{\hd-3},
\end{equation}
 by Corollary \ref{residualstress} and \eqref{ellforms}.
\end{remark}
\begin{remark}
We shall continue to use the notation \eqref{ellofr}, \eqref{uvofr} below.
\end{remark}

 We now introduce the class to which function $f$  in \eqref{sssefrep} will belong.
 For $a>0$, let us denote 
 \[
 \uu(a)=\{|u-1|\le a\}.
 \]
   Given  $M\ge0$, define
 \begin{multline*}
\cc( M)=\{f\in C^3(\uu(1/8)):
f(1)=1,\;f'(1)=0,\;f''(1)>0,\\
\|f'''\|_\infty\le M f''(1)\}.
\end{multline*}
The important parameter $f''(1)$ will appear frequently, and for convenience
we shall label it as $\beta(f)=f''(1)$.
We shall see in Lemma \ref{moduli}, that this parameter is proportional to
the shear modulus.  
The family  $\{\cc(M)\}_{M\ge0}$ is increasing with respect to $M$.
Note that for every $B>0$, there exists $f_B\in\cc(M)$ 
with $\beta(f_B)=B$, as illustrated by the  functions 
\[
f_B(u)=1+\half B(u-1)^2,
\quad B>0.
\] 
We have   discussed the assumption that $f(1) = W(I)=1$ in Section \ref{notationsec}.
 We have  also seen in Lemma \ref{constlem}
that the condition $f'(1)=0$ is necessary.  Along with $\kk(\hd)$, the positivity
of $\beta(f)$  relates to the coercivity condition for \eqref{ODE2.0}.
The restriction on the third derivative will enable us 
 to  establish estimates for the coefficients in \eqref{biguforms} uniform
 with respect to  $\beta(f)$ for any $f\in\cc(M)$ in the following lemma, and this,
in turn,   will prove essential in establishing existence of solutions.
\begin{lemma}
\label{fbODE}
Fix $\hd$ with $\kk(\hd)>0$, $M\ge0$, and let $f\in\cc(M)$.  
Define
\begin{equation}
\label{deltainterval2}
 \delta\equiv \min\{1/8,1/(8|\hd|),1/(8M)\}.
\end{equation}

Let the functions $U_i(u)$, $i=1,2$, be defined according to \eqref{biguforms}.
Then
  \begin{subequations}
\begin{equation}
\label{veedef}
\begin{aligned}
&V_1(u)=[2U_1(u)-U_2(u)]U_1(u)^{-1}\\
&V_2(u)=\left(\kappa(\hd)+\beta(f)\right)uU_1(u)^{-1}
\end{aligned}
\end{equation}
are well-defined $C^1$ functions on $\uu(\delta)$ such that
\begin{equation}
\label{v2props}
V_1(1)=0,\quad V_2(1)=1, \quad V_2(u)>0,\quad u\in\uu( \delta),
\end{equation}
and
\begin{equation}
\label{vests}
|V_i^{(j)}(u)|<C_0,\quad u\in\uu( \delta),\quad i,j=1,2,
\end{equation}
for some constant $C_0$ depending only on $\hd$ and $M$.
\end{subequations}

  Suppose that $\phi$ satisfies \eqref{phicond} and 
  $u(r)=\lambda_1(r)/\lambda_2(r)=r\phi'(r)/\phi(r)$ satisfies
  \begin{subequations}
  \begin{equation}
\label{deltainterval1}
  u(r)\in\uu(\delta),\quad r\in[0,1].
\end{equation}
Then equation \eqref{ODE2.0} is equivalent to 
\begin{multline}
\label{ODE2}
\phi''(r)+\frac2r  \left(\phi'(r)-\phi(r)/r\right)
=V_1(u(r))\frac1r \left(\phi'(r)-\phi(r)/r\right)\\
+{\mu}\left(\kk(\hd)+\beta(f)\right)^{-1}\;r\;\volume(r)^{-\hd/3+1}V_2(u(r)),
\quad r\in[0,1].
\end{multline}
\end{subequations}

If $|\hd|/\beta(f)<\delta$, then the function $g(u)$ defined in \eqref{sshBC0}
has a unique zero $u_0\in\uu(|\hd|/\beta(f))$ and $\sgn(u_0-1)=-\sgn\hd$.
If $u(1)=u_0$, then the boundary condition \eqref{sshBC0} is satisfied.
\end{lemma}

\begin{proof} 
Fix $\hd$ with $\kk(\hd)>0$ and $M\ge0$.  Let $f\in\cc(M)$,
and recall the notation $\beta(f)=f''(1)$.  
It follows from Taylor's theorem
and the definition of $\cc(M)$ that
\begin{subequations}
\begin{equation}
\label{f0123}
\begin{aligned}
&|f(u)-1-\half \beta(f)(u-1)^2|\le \tfrac16M\beta(f)|u-1|^3,\\
&|f'(u)-\beta(f)(u-1)|\le \half M\beta(f)|u-1|^2,\\
&|f''(u)-\beta(f)|\le M\beta(f)|u-1|,\\
&|f'''(u)|\le M\beta(f),
\end{aligned}
\end{equation}
for  $u\in\uu(1/8)$.  

Define the continuous function
\begin{equation*}
\mathpzc{q}(u)=
\begin{cases}
\displaystyle \frac{f'(u)}{u-1}-\beta(f),&0<|u-1|\le1/8\\
0,&u=1.
\end{cases}
\end{equation*}
Note that
\begin{equation*}
\mathpzc{q}'(u)=
\begin{cases}
\displaystyle \frac{f''(u)-\beta(f)}{u-1}-\frac{f'(u)-\beta(f)(u-1)}{(u-1)^2},
&0<|u-1|\le1/8\\
\half f'''(1),&u=1
\end{cases}
\end{equation*}
is also continuous, and thus, $\mathpzc{q}$ is $C^1$
on the interval $\uu(1/8)$.
Moreover, by 
\eqref{f0123},  we have the inequalities
\begin{equation}
\label{ff12}
\begin{aligned}
&\left|\mathpzc{q}(u)\right|\le \half M\beta(f)|u-1|,\\
&\left|\mathpzc{q}'(u)\right|\le \tfrac32M\beta(f),
\end{aligned}
\end{equation}
on $\uu(1/8)$.

We now restrict our attention to the interval $\uu(\delta)\subset\uu(1/8)$.
From \eqref{f0123}, \eqref{ff12}, and \eqref{deltainterval2}, we obtain
 the estimates
\begin{equation}
\begin{aligned}
\label{fff}
&f(u)\ge1+\half \beta(f)(u-1)^2-\tfrac16M\beta(f)\delta(u-1)^2>1,\\
&f''(u)\ge \beta(f)-M\beta(f)\delta\ge\tfrac78\beta(f),\\
&|f(u)|\le 1+\half \beta(f)(u-1)^2+\tfrac16M\beta(f)|u-1|^3\le 1+\beta(f),\\
&|f'(u)|\le \beta(f)|u-1|+\half M\beta(f)(u-1)^2\le 2\beta(f)\delta,\\
&|f''(u)|\le \beta(f)+M\beta(f)|u-1|
\le \tfrac98\beta(f),\\
&\left|\mathpzc{q}(u)\right|
\le \tfrac1{16}\beta(f),
\end{aligned}
\end{equation}
 on the interval $\uu(\delta)$.
\end{subequations}

From \eqref{fff}, we obtain the coercivity estimate
\begin{subequations}
\begin{equation}
\label{bigulbd}
\begin{aligned}
U_1(u)\ge&\kk(\hd)+u^2f''(u)-\tfrac23|\hd||uf'(u)|\\
\ge &\kk(\hd)+(1-\delta)^2f''(u)-\tfrac23|\hd|(1+\delta)|f'(u)|\\
\ge&\kk(\hd)+(\tfrac78)^3\beta(f)-(\tfrac23)(\tfrac98)(2|\hd|\delta)\beta(f)\\
\ge&\kk(\hd)+(\tfrac78)^3\beta(f)-\tfrac3{16}\beta(f)\\
\ge&\kk(\hd)+\tfrac14\beta(f),
\end{aligned}
\end{equation}
on $\uu(\delta)$.

It follows that the  functions $V_i(u)$, $i=1,2$ in \eqref{veedef}
are well-defined and $C^1$ on $\uu(\delta)$.
Therefore, the ODEs
\eqref{ODE2.0}  and \eqref{deltainterval1}, \eqref{ODE2}
are equivalent.
 By inspection, 
 \eqref{v2props}   holds.

Using \eqref{f0123}, \eqref{ff12}, \eqref{fff}, it is straightforward to verify that the functions
$U_i(u)$, $i=1,2$, defined in \eqref{biguforms} satisfy the bounds
\begin{equation}
\label{bigubds}
|U_i^{(j)}(u)|\le C(1+\beta(f)),\quad u\in\uu(\delta), \quad i,j=1,2.
\end{equation}
 \end{subequations}
 The constant depends on $\hd$ and $M$, but not $\beta(f)$.

With the aid of  \eqref{bigulbd}, \eqref{bigubds}, we can now verify that
the estimates \eqref{vests} also hold.

Finally, we prove to the statement concerning
the function $g$ defined in \eqref{sshBC0}.
Returning to \eqref{fff}, we have
\begin{multline*}
\|g''\|_{L^\infty(\uu(\delta))}
=\left\|\left(\tfrac\hd3+2\right)f''(u)+uf'''(u)\right\|_{L^\infty(\uu(\delta))}\\
\le \tfrac98\left({|\hd|}+2\right)\beta(f)+\tfrac98 M\beta(f)
= \tfrac9{8}\beta(f)\left({|\hd|}+2+M\right).
\end{multline*}
Application of Taylor's theorem yields
\begin{align*}
|g(u)-\hd/3-\beta(f)(u-1)|&=|g(u)-g(1)-g'(1)(u-1)|\\
&\le \half  \|g''\|_{L^\infty(\uu(\delta))} (u-1)^2\\
&\le\tfrac9{16}\beta(f)\left({|\hd|}+2+M\right)\delta|u-1|\\
&\le \tfrac9{16}\beta(f)\left(\tfrac1{8}+\tfrac14+\tfrac18\right)|u-1|\\
&\le\tfrac12\beta(f)|u-1|, 
\end{align*}
on $\uu(\delta)$, and
\begin{align*}
|g'(u)-\beta(f)|&=|g'(u)-g'(1)|\\
&\le\|g''\|_{L^\infty(\uu(\delta))}|u-1|\\
&\le\tfrac9{8}\beta(f)\left({|\hd|}+2+M\right)\delta\\
&<\tfrac34\beta(f),
\end{align*}
on $\uu(\delta)$.

From the first of these inequalities, it follows that if $|\hd|/\beta(f)<\delta$, then
\begin{align*}
&g(u)
>0,\quad|\hd|/\beta(f)<u-1<\delta,\\
&g(u)<0,\quad-\delta<u-1<-|\hd|/\beta(f).
\end{align*}
Thus, $g$ has a zero $u_0\in\uu(|\hd|/\beta(f))$.
It follows from the second inequality, that $g$ is strictly increasing
on $\uu(\delta)$ and since $g(1)=\hd/3$, that $\sgn(u_0-1)=-\sgn\hd$.

\end{proof}

\begin{remark}
\label{homnullLag}
The homogeneous solutions of \eqref{nullLag} are given by 
\begin{equation*}
\llll^0(\lambda)
=c_0\lambda_2^{\hd} \left(\tfrac\hd3\lambda_1/\lambda_2+\left(1-\tfrac\hd3\right)\right)
=c_0 v^{\hd/3}u^{-\hd/3}\left(\tfrac\hd3u+\left(1-\tfrac\hd3\right)\right),
\end{equation*}
for any  $c_0,\hd \in\rr$.
Thus, spherically symmetric null Lagrangians may be homogeneous of any degree
in $F$.
Null Lagrangians are necessarily homogeneous of degree $\hd=1$, $2$,  or $3$ in $F$,
see \cite{Ball-1981}.
For example,  the classical null Lagrangian  $W^0(F)=\det F$
has $\hd=3$ and $\llll^0(\lambda)=\det F(\lambda,\omega)=
\lambda_1\lambda_2^2=v$.
\end{remark}

\section{Existence of  eigenfunctions}
\label{efunctions}

We shall now address the question of existence of solutions to the
problem \eqref{ODE2},  
\eqref{sshBC0}.

\begin{theorem}
\label{th1}
Fix $\hd$ with $\kk(\hd)>0$, $M\ge0$, and
let $f\in \cc(M)$.
There exists a  small constant $R>0$
depending only on $\hd$ and $M$
 such that if
\[
|\mu|/(\kappa(\hd)+\beta(f))\le R,
\]
 then equation \eqref{ODE2} has a solution $\phi^\mu\in C^2([0,1])$ satisfying
 \begin{subequations}
\begin{equation}
\label{IVP2}
\phi^\mu(0)=D^2_r\phi^\mu(0)=0,\quad 
D_r\phi^\mu(0)=1,
\end{equation}
as well as the estimates
\begin{equation}
\label{phibmest1}
|\phi^\mu(r)/r-1|,\;|D_r\phi^\mu(r)-1|,\;|u^\mu(r)-1|
\le Rr^2,
\end{equation}
and
\begin{equation}
\label{phibmest2}
|D_r^2\phi^\mu(r)|\le Rr,
\end{equation}
for $r\in[0,1]$.  
\end{subequations}

The map from $\{\mu: |\mu|/(\kappa(\hd)+\beta(f))\le R\}$ to $ C([0,1])$ given by
\[
\mu
\mapsto \phi^\mu
\]
is continuous.

If $\mu>0$, then 
\begin{subequations}
\begin{equation}
\label{princstr1}
\lambda_1^\mu(r)=D_r\phi^\mu(r)>\lambda_2^\mu(r)=\phi^\mu(r)/r>1,\quad r\in(0,1],
\end{equation}
and if $\mu<0$, then
\begin{equation}
\label{princstr2}
\lambda_1^\mu(r)<\lambda_2^\mu(r)<1,\quad r\in(0,1].
\end{equation}
\end{subequations}

If 
$\beta(f)$ is sufficiently large, then there exists an 
eigenvalue $\mu\ne0$ with $\sgn\mu=-\sgn\hd$
  such that the
solution $\phi^\mu$  satisfies the boundary condition \eqref{sshBC0}.  
\end{theorem}

\begin{remark}
The assumption that $D_r\phi^\mu(0)=1$ in
\eqref{IVP2} does not restrict the possible initial data of the motion
 in \eqref{separability0}.
\end{remark}
\begin{proof}[Proof of Theorem \ref{th1}.]
In order to handle the apparent singularity at $r=0$, it is convenient to make
the ansatz
\begin{equation}
\label{ansatz}
\phi(r)=r+K\zeta(r)\equiv r+\int_0^r(r-\rho)\rho\zeta(\rho)d\rho,\isp{with}\zeta\in C([0,1]).
\end{equation}
Notice that $K$ is a bounded linear operator from $C([0,1])$
into $C^2([0,1])$.   
In fact, it follows from \eqref{ansatz} that
\begin{subequations}
\begin{equation}
\label{zetaformulas}
\begin{aligned}
&\phi'(r)=\lambda_1(r)=1+\int_0^r\rho\zeta(\rho)d\rho\\
&\phi(r)/r=\lambda_2(r)=1+\frac1r\int_0^r(r-\rho)\rho\zeta(\rho)d\rho\\
&\phi'(r)-\phi(r)/r=\lambda_1(r)-\lambda_2(r)
=\frac1r\int_0^r\rho^2\zeta(\rho)d\rho\\
&\phi''(r)=r\zeta(r).
\end{aligned}
\end{equation}
In particular, $\phi(r)=r+K\zeta(r)$ 
satisfies the conditions \eqref{IVP2}.
 
 Assume now that  \eqref{ansatz} holds with
 \[
 \zeta\in N_R=\{\zeta\in C([0,1]):\zinf<R\},\quad  R\le\delta,
 \]
 where $\delta=\min\{1/8,1/(8|\hd|),1/(8M)\}$ was previously defined in \eqref{deltainterval2}.
 
 Straight-forward pointwise  estimates for $r\in[0,1]$
 yield
\begin{equation}
\label{zetaests}
\begin{aligned}
&|\phi(r)-r|\le \tfrac16\zinf r^3\le \tfrac16 Rr^3\\
&|\phi'(r)-1|=|\lambda_1(r)-1|<\half\zinf r^2\le \half Rr^2\\
&|\phi(r)/r-1|=|\lambda_2(r)-1|\le \tfrac16\zinf r^2\le \tfrac16 Rr^2\\
&|\lambda_1(r)-\lambda_2(r)|\le \third\zinf r^2\le \third Rr^2\\
&\lambda_2(r)\ge1-|\lambda_2(r)-1|\ge 1-\tfrac16Rr^2\ge2/3\\
&|u(r)-1|=\left|\lambda_2(r)^{-1}(\lambda_1(r)-\lambda_2(r))\right|\le \tfrac12\zinf r^2\le\half Rr^2\\
&|\volume(r)-1|=\left|\lambda_1(r)\lambda_2(r)^2-1\right|\le  2\zinf r^2\le  2Rr^2.
\end{aligned}
\end{equation}
\end{subequations}
By \eqref{zetaests}, it follows that  \eqref{phibmest1}, \eqref{phibmest2} hold,
and as a consequence \eqref{phicond}, \eqref{deltainterval1} also are
valid.

Since $\zeta\in N_R$ is small, we  regard $\phi(r)=r+K\zeta(r)$ as a perturbation of
the identity map.  Explicitly, 
$\zeta(s)\equiv0$ implies that
\[
\phi(r)=r\isp{and}\lambda_1(r)=\lambda_2(r)=u(r)=\volume(r)=1.
\]
Note that  $\phi(r)=r$ solves \eqref{ODE2} with $\mu=0$.

 Making the substitution \eqref{ansatz} in equation \eqref{ODE2} and using
 \eqref{zetaformulas}, we find that
\begin{equation}
\label{ODE3}
L\zeta(r)=\ffzbm(r),
\end{equation}
with
\begin{equation*}
\label{isomorphism}
L\zeta(r)= \zeta(r)+\frac2{r^3}\int_0^r\rho^2\zeta(\rho)d\rho
\end{equation*}
 and
 \begin{multline*}
\ffzbm(r)=V_1(u(r))
\frac1{r^3}\int_0^r\rho^2\zeta(\rho)d\rho\\
+\mu(\kappa(\hd)+\beta(f))^{-1}\volume(r)^{-\hd/3+1}V_2(u(r)).
\end{multline*}
Recall that the functions $V_i(u)$ depend on $f\in\cc(M)$
and satisfy the conditions  \eqref{v2props}, \eqref{vests}.

The operator $L$ 
is an isomorphism on $C([0,1])$ with bounded inverse
\begin{equation*}
L^{-1}\eta(r)=\eta(r)-\frac2{r^5}\int_0^r\rho^4\eta(\rho)d\rho.
\end{equation*}
From \eqref{ODE3}, 
we arrive at the reformulation
\begin{equation}
\label{ODE4}
\zeta(r)=L^{-1}\ffzbm(r).
\end{equation}
In order to solve  \eqref{ODE2}, \eqref{IVP2}, it is sufficient to find a 
solution $\zeta$ of \eqref{ODE4} in $N_R$, with $R\le\delta $.

Let us define
\begin{equation*}
\eps=\mu(\kappa(\hd)+\beta(f))^{-1},
\end{equation*}
 since this expression  appears repeatedly.
 
  The claim is that for $|\eps|\le R\ll1$ the map
\begin{equation*}
\zeta\mapsto L^{-1}\ffzbm
\end{equation*}
is a uniform contraction on $N_R$ taking $N_R$ into itself.

Assume that 
\begin{equation}
\label{paramass}
|\eps|\le R\le \delta .
\end{equation}
As a consequence of \eqref{v2props}, \eqref{vests}, \eqref{zetaests}, 
 there exists a constant $C_1$, 
independent of $R$ and $\beta(f)$,
 such that 
 \begin{subequations}
 \begin{multline}
\label{cbd}
\|\ff(\zeta_1,\mu)-\ff(\zeta_2,\mu)\|_\infty\\
\le C_1 \left(R+\eps\right)\|\zeta_1-\zeta_2\|_\infty
\le 2C_1 R\|\zeta_1-\zeta_2\|_\infty,
\end{multline}
for all $\zeta_1,\zeta_2\in N_R$.
By \eqref{v2props}, we have 
$\ff(0,\mu)(r)=\eps$.
It follows from \eqref{cbd} that 
\begin{equation}
\label{ballbd}
\|\ffzbm-\eps\|_\infty\le 2C_1R^2,
\end{equation}
for all $\zeta\in N_R$.
\end{subequations}

For the contraction estimate,  we have from \eqref{cbd}
\begin{subequations}
\begin{multline}
\label{contr1}
\|L^{-1}\ff(\zeta_1,\mu)-L^{-1}\ff(\zeta_2,\mu)\|_\infty\\
\le \|L^{-1}\|\;2C_1R\;\|\zeta_1-\zeta_2\|_\infty
\le (1/10)\|\zeta_1-\zeta_2\|_\infty,
\end{multline}
 for all  $\zeta_1,\zeta_2\in N_R$, provided $\|L^{-1}\|2C_1R\le 1/10$.

Since $L^{-1}\eps=3\eps/5$
for any constant function $\eps$,
 by \eqref{ballbd}, \eqref{paramass}, there holds
\begin{multline}
\label{contr2}
\|L^{-1}\ffzbm\|_\infty
=\|L^{-1}(\ffzbm-\eps)+3\eps/5\|_\infty\\
\le \|L^{-1}\|2C_1R^2+3\eps/5\le R/10+3R/5 < R,
\end{multline}
\end{subequations}
for all $\zeta\in N_R$, provided $\|L^{-1}\|2C_1R\le1/10$.
This shows that the map leaves $N_R$ invariant.

This establishes the claim
under the restrictions
\begin{equation}
\label{rcond}
R\le\delta\isp{and} 2C_1\|L^{-1}\|R\le1/10.
\end{equation}
We emphasize that these restrictions on $R$ do not depend on the value of $\beta(f)$,
and therefore, the estimates \eqref{contr1}, \eqref{contr2} hold for all $f\in\cc(M)$.

By the uniform contraction principle (see \cite{Chow-Hale-1982} Section 2.2,
or \cite{Sideris-2013} Section 5.3),  
the equation \eqref{ODE4}
has a unique solution $\zeta^\mu\in N_R\subset C([0,1])$, for each 
$\mu=\eps(\kappa(\hd)+\beta(f))$
such that
$|\eps|\le R$.
Moreover, the map 
$\mu\mapsto \zeta^\mu$, from $\{\mu:|\eps|<R\}$ to $C([0,1])$,  is continuous.
In particular, $\zeta^0(r)\equiv0$, by uniqueness.

For each such $\zeta^\mu$, we
obtain a $C^2$ solution 
\[
\phi^\mu(r)=r+K\zeta^\mu(r),\quad \mu=\eps(\kappa(\hd)+\beta(f)),\quad |\eps|\le R,
\]
 of \eqref{ODE2} which
satisfies 
 \eqref{IVP2}, \eqref{phibmest1}, \eqref{phibmest2} and 
 depends continuously on the parameter $\mu$.\footnote{
 The map $(\mu,f)\mapsto\phi^{\mu,f}$, from the open set
 $\{(\mu,f): |\eps|<R,\; f\in\cc(M)\}$
 in $\rr\times\cc(M)$ with the product topology to $C([0,1])$, is continuous.}

Since, by definition \eqref{ellofr},
\begin{equation}
\label{ell2ode}
D_r\lambda_2^\mu(r)=\frac1r\left(\lambda_1^\mu(r)-\lambda_2^\mu(r)\right),
\end{equation}
 it follows from \eqref{ODE2} that
 \begin{multline*}
D_r\left(\lambda_1^\mu(r)-\lambda_2^\mu(r)\right)
+\frac3r\left(\lambda_1^\mu(r)-\lambda_2^\mu(r)\right)
=V_1\left(u^\mu(r)\right)\frac1r 
\left(\lambda_1^\mu(r)-\lambda_2^\mu(r)\right)\\
+{\mu}\left(\kk(\hd)+\beta(f)\right)^{-1}\;r\;
\volume^\mu(r)^{-\hd/3+1}
V_2\left(u^\mu(r)\right).
\end{multline*}
This, in turn, may be expressed in the form
\begin{equation*}
D_r\mathcal X^\mu(r)=\mathcal Y^\mu(r)\mathcal X^\mu(r)+\mu\mathcal Z^\mu(r),
\end{equation*}
where
\begin{align*}
&\mathcal X^\mu(r)=r^3\left(\lambda_1^\mu(r)-\lambda_2^\mu(r)\right)\\
&\mathcal Y^\mu(r)=\frac1rV_1\left(u^\mu(r)\right) \\
&\mathcal Z^\mu(r)=\left(\kk(\hd)+\beta(f)\right)^{-1}\;r^4\;
\volume^\mu(r)^{-\hd/3+1}
V_2\left(u^\mu(r)\right).
\end{align*}
Note that, by \eqref{v2props}, \eqref{zetaests},  
the coefficients $\mathcal Y^\mu$, $\mathcal Z^\mu$ are continuous on $[0,1]$
and $\mathcal Z^\mu$ is strictly positive on $(0,1]$, by \eqref{v2props}.
So we can write
\begin{equation*}
\mathcal X^\mu(r)=\mu \int_0^r\left(\exp\int_0^\rho\mathcal Y^\mu(s)ds\right)\mathcal Z^\mu(\rho)d\rho.
\end{equation*}
We conclude that if $\mu>0$, then $\mathcal X^\mu$ is strictly positive on $(0,1]$,
and then from \eqref{ell2ode}, that $D_r\lambda_2^\mu(r)$
is strictly positive on $(0,1]$.  Thus, from \eqref{IVP2}, we have
\begin{equation*}
\lambda_1^\mu(r)>\lambda_2^\mu(r)>\lambda_2^\mu(0)=D_r\phi^\mu(0)=1,
\quad r\in(0,1].
\end{equation*}
This proves \eqref{princstr1}, and \eqref{princstr2} follows analogously.

For future reference, we also record the fact that 
\begin{equation}
\label{sgnrel}
\sgn(u^\mu(1)-1)=\sgn \mathcal X^\mu(1)=\sgn\mu.
\end{equation}

We shall now show that for all $\hd$ with $\kk(\hd)>0$ and  $\beta(f)$ sufficiently large,
the eigenvalue $\mu$ may be chosen so that 
 the boundary condition \eqref{sshBC0}  
 is fulfilled.

 We continue to assume that $R$ satisfies the conditions \eqref{rcond}.
  Define  
  \[
  \mu(\eps)=\eps(\kappa(\hd)+\beta(f)),
\quad |\eps|\le R,
\]
 and consider the solution family
\[
\{\phi^{\mu(\eps)}(r)=r+K\zeta^{\mu(\eps)}(r):{|\eps|\le R}\}.
\]

The first step will be to establish lower bounds for $|u^{\mu(\pm R)}(1)-1|$.
  Since $\zeta^{\mu(\eps)}$ solves \eqref{ODE4}, 
 we have  from \eqref{ballbd}, \eqref{rcond},
\begin{multline*}
\|\zeta^{\mu(\eps)}-3\eps/5\|_\infty
=\|L^{-1}(\ff(\zeta^{\mu(\eps)},\mu(\eps))-\eps)\|_\infty\\
\le \|L^{-1}\|2C_1R^2\le R/10.
\end{multline*}
Letting $\eps=\pm R$, this implies that
\begin{equation}
\label{zlb}
\zeta^{\mu(R)}(r)\ge R/2\isp{and}\zeta^{\mu(-R)}(r)\le -R/2,\isp{for} r\in[0,1].
\end{equation}

From   \eqref{zetaformulas}, \eqref{zetaests}, \eqref{zlb}, we obtain
\begin{equation*}
\begin{aligned}
&\lambda_1^{\mu(R)}(1)-\lambda_2^{\mu(R)}(1)
=\int_0^1\rho^2\zeta^{\mu(R)}(\rho)d\rho\ge (R/2)(1/3)=R/6\\
&\lambda_1^{\mu(-R)}(1)-\lambda_2^{\mu(-R)}(1)
=\int_0^1\rho^2\zeta^{\mu(-R)}(\rho)d\rho\le-R/6\\
&1<\lambda_2^{\mu(R)}(1)=1+\int_0^1(1-\rho)\rho \zeta^{\mu(R)}(\rho)d\rho\le 1+R/6<7/6\\
&1>\lambda_2^{\mu(-R)}(1)=1+\int_0^1(1-\rho)\rho \zeta^{\mu(-R)}(\rho)d\rho\ge 1-R/6>5/6.
\end{aligned}
\end{equation*}

These estimates combine to show that
\begin{equation}
\label{ulb}
\begin{aligned}
&u^{\mu(R)}(1)-1
=\frac{\lambda_1^{\mu(R)}(1)-\lambda_2^{\mu(R)}(1)}
{\lambda_2^{\mu(R)}(1)}
\ge\frac{R/6}{7/6}=R/7\\
&u^{\mu(-R)}(1)-1
=\frac{\lambda_1^{\mu(-R)}(1)-\lambda_2^{\mu(-R)}(1)}
{\lambda_2^{\mu(-R)}(1)}
\le\frac{-R/6}{5/6}<-R/7.
\end{aligned}
\end{equation}

Suppose now that $f\in\cc(M)$ with $\beta(f)$ sufficiently large:
\begin{equation}
\label{betacond}
\beta(f)\ge7|\hd|/R.
\end{equation}
Then
\[
|\hd|/\beta(f)<R\le\delta,
\]
so that by Lemma \ref{fbODE}, $g$ has a unique zero $u_0\in\uu(|\hd|/\beta(f))\subset\uu(R)$.
By continuous dependence upon parameters, the function
\[
z(\eps)=u^{\mu(\eps)}(1)-1
\]
 is continuous for $|\eps|\le R$.
 Using \eqref{ulb}, \eqref{betacond}, we find that
 \[
 z(R)>R/7>|\hd|/\beta(f)>u_0-1>-|\hd|/\beta(f)>-R/7>z(-R).
 \]
 So there exists $|\eps_0|<R$ such that $z(\eps_0)=u_0-1$,
 and thus, 
 \[
 g(u^{\mu(\eps_0)}(1))=g(u_0)=0.
 \]
 By Lemma \ref{fbODE} and \eqref{sgnrel}, we also have that
 \begin{equation*}
-\sgn\hd=\sgn(u_0-1)
=\sgn(u^{\mu(\eps_0)}(1)-1)
=\sgn\mu(\eps_0).
\end{equation*}

Therefore, we have shown that if
the eigenvalue is taken to be 
$\mu(\eps_0)=\eps_0(\kk(\hd)+\beta(f))$,
then $\phi^{\mu(\eps_0)}$ satisfies
the boundary condition \eqref{sshBC0} 
  and 
 $\sgn\mu(\eps_0)=-\sgn\hd$.

\end{proof}

\begin{remark}
The boundary condition
$u^{\mu(\eps_0)}(1)=u_0$  is equivalent to
 the linear and homogeneous Robin boundary condition
\[
D_r\phi^{\mu(\eps_0)}(1)=u_0\phi^{\mu(\eps_0)}(1).
\]
\end{remark}

\begin{corollary}
\label{defexist}
Fix $\hd$ with $\kk(\hd)>0$ and $M\ge0$.  
Suppose  that $W$ satisfies
\eqref{sef}, \eqref{objective}, \eqref{isotropic}, \eqref{homogeneity}, and
using Lemma \ref{constlem},
\begin{equation*}
W\circ F(\lambda,\omega)=\volume^{\hd/3}f(u),
\isp{with} f\in\cc(M),
\end{equation*}
for all $\lambda\in\rr_+^2$ and $\omega\in S^2$.
 Let $S$ be  defined by \eqref{Piola}.

If 
 $\beta(f)$ is sufficiently large,
 then there exists a $C^2$ spherically symmetric orientation-preserving 
deformation $\varphi:\bbbar\to\rr^3$ and an eigenvalue $\mu$ with $\sgn\mu=-\sgn \hd$
satisfying \eqref{elliptic1}, \eqref{BC1}.

If $\hd<0$, then $\varphi(\bbbar)\supset\bbbar$
and the principal stretches
satisfy $\lambda_1(r)>\lambda_2(r)>1$, $r\in(0,1]$, while when $\hd>3$,
the inclusion and  inequalities are reversed.
\end{corollary}

\begin{proof}
If $\beta(f)$ is sufficiently large,  Theorem \ref{th1} 
ensures the existence of an eigenvalue  $\mu\in\rr$, with $\sgn\mu=-\sgn\hd$,
and 
an eigenfunction function $\phi=\phi^\mu\in C^2([0,1])$ satisfying
the hypotheses of Lemma \ref{fbODE}, the differential 
equation \eqref{ODE2}, and the boundary condition
\eqref{sshBC0}.  Therefore, $\phi$ also solves the differential equation \eqref{ODE2.0}.

By assumption, $f(u)$ is the restriction of an appropriate strain energy function to
the spherically symmetric deformation gradients.  Therefore, Lemma \ref{evp3}
yields the desired deformation $\varphi(y)=\phi(r)\omega$.

Since $\varphi(\bbbar)$ is a sphere of radius $\phi(1)=\lambda_2(1)$,
the final statements follow from \eqref{princstr1}, \eqref{princstr2}.
\end{proof}

\begin{remark}
Since we have assumed that the reference density is $\br=1$, the density of the
deformed configuration in material coordinates is 
$\varrho\circ\varphi(y)=v(r)^{-1}=\left(\lambda_1(r)\lambda_2(r)^2\right)^{-1}$.
If $\hd<0$, we have $\varrho\circ\varphi(y)<1$, for $r\in(0,1]$, and if $\hd>3$, we have
$\varrho\circ\varphi(y)>1$, for $r\in(0,1]$.
\end{remark}

\section{Existence of expanding and collapsing bodies}
\label{mainresult}
Putting together the results of Lemmas \ref{separable} and \ref{dyn} with Corollary \ref{defexist}
yields our main result.

\begin{theorem}
\label{mainthm} 
Under the assumptions of Corollary \ref{defexist},
let $\varphi:\bbbar\to\rr^3$ be the resulting  $C^2$ spherically symmetric 
orientation-preserving deformation solving 
\eqref{elliptic1}, \eqref{BC1} with corresponding eigenvalue $\mu$,
such that $\sgn\mu=-\sgn\hd$.

Given the eigenvalue $\mu$, let $\tdsf:[0,\tau)\to\rr_+$ be the 
$C^2$ solution of the initial value problem
for \eqref{ODE1} with initial data $(\tdsf(0),\dot \tdsf(0))\in\rr_+\times\rr$,
from Lemma \ref{dyn}.

Then by Lemma \ref{separable}, $x(t,y)=\tdsf(t)\varphi(y)$ is a  motion 
in $C^2([0,\tau)\times\bbbar)$ satisfying \eqref{elasticityPDE}, \eqref{elasticityBC}.
Under this motion, the spatial configuration  of the elastic body at time $t$ is a sphere $\omt$
of radius $\tdsf(t)\phi(1)$.  

If $\hd<0$, then the lifespan of the solution satisfies
$\tau=+\infty$ and $0<(2E(0))^{1/2}-\tdsf(t)/t\to0$, as $t\to\infty$,
where $E(0)=\half\dot\tdsf(0)^2-\tfrac{\mu}{\hd}a(0)^{\hd}$.

If $\hd>3$, then $\tau<\infty$ and $\tdsf(t)\to0$, as $t\to\tau$.

\end{theorem}

\begin{remark}
By  \eqref{respress}, the  sign of the residual pressure 
$\mathcal P(1)=-\hd/3$ determines
whether the body expands or collapses.
\end{remark}

\section{Constitutive theory}
\label{constsec}

The next result is a partial converse to Lemma \ref{constlem}.

\begin{theorem}
\label{thm2}
If $f:\rr_+\to\rr_+$ 
is smooth, with  $f'(1)=0$, then there exists 
a $C^2$  strain energy function $W$
satisfying \eqref{sef}, \eqref{objective}, \eqref{isotropic}, \eqref{homogeneity},
  such that \eqref{sssefrep} holds.
  \end{theorem}

\begin{remark}
Note that  $f$ is required to be strictly positive and sufficiently differentiable.
\end{remark}

\begin{proof}[Proof of Theorem \ref{thm2}.]
The construction is motivated by
\eqref{oihsef}, \eqref{oihsef1}.

Recall from Lemma \ref{flomdef} that $F(\lambda,\omega)=A(F(\lambda,\omega))$ is positive definite and symmetric
with eigenvalues $\lambda_1,\lambda_2,\lambda_2$.
So from \eqref{invsst}, \eqref{lamuinv}, we have
\begin{subequations}
\begin{equation}
\label{ssinvariants1}
\det(F(\lambda,\omega))=\lambda_1\lambda_2^2=\volume,\\
\end{equation}
\begin{equation}
\label{ssinvariants2}
\begin{aligned}
H_1(\Sigma(F(\lambda,\omega)))&= \tfrac13\tr\Sigma(F(\lambda,\omega))\\
&=\third \det(F(\lambda,\omega))^{-1/3}\tr F(\lambda,\omega)\\
&=\third (\lambda_1\lambda_2^2)^{-1/3}(\lambda_1+2\lambda_2)\\
&=\third(u^{2/3}+2u^{-1/3})\\
&\equiv h_1(u),
\end{aligned}
\end{equation}
and since $\Sigma(F(\lambda,\omega))\in\slt$, 
\begin{equation}
\label{ssinvariants3}
\begin{aligned}
H_2(\Sigma(F(\lambda,\omega)))
&=\third\tr\cof\Sigma(F(\lambda,\omega))\\
&=\third\tr\Sigma(F(\lambda,\omega))^{-1}\\
&=\third \det(F(\lambda,\omega))^{1/3}\tr F(\lambda,\omega)^{-1}\\
&=\third(\lambda_1\lambda_2^2)^{1/3}(\lambda_1^{-1}+2\lambda_2^{-1})\\
&
=\third(u^{-2/3}+2u^{1/3})\\
&\equiv h_2(u).
\end{aligned}
\end{equation}
\end{subequations}

It is enough to find a $C^2$ function
$\Phi:\rr_+^2\to\rr_+$ such that 
\begin{equation}
\label{claim}
\Phi(h_1(u),h_2(u))=f(u), \quad u\in\rr_+,
\end{equation}
for then, by  \eqref{ssinvariants1}, \eqref{ssinvariants2}, \eqref{ssinvariants3}, 
the function
\begin{equation*}
W(F)=(\det F)^{\hd/3}\Phi(H_1(\Sigma(F)),H_2(\Sigma(F)))
\end{equation*}
automatically satisfies the requirements \eqref{sef},
\eqref{objective}, \eqref{isotropic}, \eqref{homogeneity}, \eqref{sssefrep}.
The construction of $\Phi$  is not entirely routine because the curve 
\begin{equation}
\label{bdry}
\mathcal H=\{\left(h_1(u),h_2(u)\right)):u\in\rr_+\}
\end{equation}
 has a cusp at $u=1$,
as shown in Figure \ref{region}.

Equivalently, \eqref{claim} will be proven if we can find a $C^2$ function $\WT :\rr_+^2\to\rr_+$
such that
\[
\WT(\ell_1(u),\ell_2(u))=f(u),\quad u\in\rr_+,
\]
where
\[
\ell_1(u)=h_1(u)-1,
\quad \ell_2(u)=h_1(u)-h_2(u),
\]
because we can then simply take
\[
\Phi(x_1,x_2)=\WT\left(x_1-1, x_1-x_2\right),
\quad x\in\rr^2.
\]

Observe that  $\ell_1,\ell_2\in C^\infty(\rr_+)$ and
\begin{subequations}
\begin{equation}
\label{ellderiv}
\begin{aligned}
&\ell_1(1)=\ell_1'(1)=0,& \ell_1''(1)>0,\\
& \ell_2(1)=\ell_2'(1)=\ell_2''(1)=0,
&\ell_2'''(1)>0,
\end{aligned}
\end{equation}
so that
\begin{align}
\label{g2ests}
&|\ell_2^{(k)}(u)|\sim |u-1|^{3-k},\quad k=0,1,2,\quad |u-1|\ll1.
\end{align}

The function $\ell_2:\rr_+\to\rr$ is a homeomorphism, and
it can be written as
\begin{equation}
\label{factorization}
\ell_2(u)=\hat\ell_2(u)(u-1)^3,
\end{equation}
\end{subequations}
where $\hat\ell_2$ is a smooth positive function.

Let $\xi=\ell_2^{-1}$.  Then by \eqref{g2ests},
\[
|w|=|\ell_2\circ \xi(w)|\sim |\xi(w)-1|^3,\quad |w|\ll1,
\]
and so
\begin{subequations}
\begin{equation}
\label{etaests1}|\xi(w)-1|\sim|w|^{1/3},\quad |w|\ll1.
\end{equation}
Now $\xi \in C^2(\rr\setminus\{0\})$, and for $w\ne0$, we have
\[
\xi '(w)=1/\ell_2'\circ \xi (w),\quad \xi ''(w)=-\ell_2''\circ \xi (w)/\left(\ell_2'\circ \xi (w)\right)^3.
\]
Thus, for $0<|w|\ll1$, we obtain from \eqref{g2ests}, \eqref{etaests1},
\begin{equation}
\label{etaests2}
\begin{aligned}
&|\xi'(w)|\sim |\xi (w)-1|^{-2}\sim|w|^{-2/3},\\
&|\xi ''(w)|\sim |\xi (w)-1|^{-5}\sim|w|^{-5/3}.
\end{aligned}
\end{equation}
\end{subequations}

For $u\in\rr_+$, define the smooth function
\[
g(u)=f(u)-f(1)-g_1 \ell_1(u)-g_2 \ell_2(u)-g_3\ell_1(u)^2-g_4\ell_1(u)\ell_2(u),
\]
with $\{g_k\}_{k=1}^4$ to be determined.
 Using the hypothesis $f'(1)=0$ and 
 \eqref{ellderiv}, we derive
 \begin{align*}
 g(1)=&0\\
 g'(1)=&0\\
 g^{(2)}(1)=&f^{(2)}(1)-g_1\ell_1^{(2)}(1)\\
g^{(3)}(1)=&f^{(3)}(1)-g_1\ell_1^{(3)}(1)-g_2\ell_2^{(3)}(1)\\
g^{(4)}(1)=&f^{(4)}(1)-g_1\ell_1^{(4)}(1)-g_2\ell_2^{(4)}(1)-6g_3\ell_1^{(2)}(1)^2\\
g^{(5)}(1)=&f^{(5)}(1)-g_1\ell_1^{(5)}(1)-g_2\ell_2^{(5)}(1)-20g_3\ell_1^{(2)}(1)\ell_1^{(3)}(1)\\
&\phantom{f^{(5)}(1)-g_1\ell_1^{(5)}(1)-g_2\ell_2^{(5)}(1)}-10g_4 \ell_1^{(2)}(1)\ell_2^{(3)}(1).
\end{align*}
The system $g^{(k)}(1)=0$, $k=2,\ldots,5$
is diagonal, and since $\ell_1^{(2)}(1),\ell_2^{(3)}(1)\ne0$,  
 there exist unique values $\{g_k\}_{k=1}^4$
 such that 
 \begin{equation*}
\label{geecoeffs}
g^{(k)}(1)=0,\quad k=0,\ldots,5.
\end{equation*}
 Now $g$ is smooth, so  there exists $G\in C^2(\rr_+)$ such that
\[
g(u)=(u-1)^6G(u).
\]

By \eqref{factorization},
we may write
\begin{equation*}
g(u)
=\ell_2(u)^2\hat\ell_2(u)^{-2} G(u)
\equiv\ell_2(u)^2\widehat G(u),
\end{equation*}
with $\widehat G=(\hat\ell_2)^{-2}\cdot G\in C^2(\rr_+)$.  This, in turn, may be expressed as
\[
g(u)=\widetilde G\circ \ell_2(u),\isp{with}
\widetilde G(w)=w^2 \widehat G\circ \xi(w),\quad w\in\rr.
\]
Although $\xi$ is not differentiable at $w=0$,
it follows from \eqref{etaests1}, \eqref{etaests2} that $\widetilde G\in C^2(\rr)$.

For $x\in\rr^2$, define
\[
\WT(x)=f(1)+g_1x_1+g_2x_2+g_3 x_1^2+g_4x_1x_2+\widetilde G(x_2).
\]
Then $\WT$ is in $ C^2$ and $\WT(\ell_1(u),\ell_2(u))=f(u)$,
as desired.

Since $f>0$, by assumption,
the resulting function $\Phi$  is positive
in a neighborhood of the curve $\mathcal H$, and it can modified away from $\mathcal H$,
if necessary, to ensure positivity on its entire domain without changing its 
values along the curve $\mathcal H$.
\end{proof}

\subsection*{Example}

Going back to the example  \eqref{example} given in the Introduction,
we have using \eqref{invsst} that
$
W(\Sigma(F))=\Phi(H_1(\Sigma(F)),H_2(\Sigma(F))),
$
with
\begin{align*}
\Phi(x_1,x_2)&=1+c_1(x_1-1)+c_2(x_2-1)\\
&=1+(c_1+c_2)(x_1-1)-c_2(x_1-x_2),\quad c_1,\;c_2>0.
\end{align*}
  Thus, using \eqref{sefrestrss} and  the notation of the previous proof, 
we find
\begin{align*}
f(u)&=\Phi(H_1(\Sigma(F)),H_2(\Sigma(F)))\Big|_{F=F(\lambda,\omega)}\\
&=1+(c_1+c_2)(h_1(u)-1)-c_2(h_1(u)-h_2(u))\\
&=1+(c_1+c_2)\ell_1(u)-c_2 \ell_2(u).
\end{align*}
By \eqref{ellderiv}, we have $\beta(f)=(c_1+c_2)\ell_1''(1)>0$,
and 
$f\in\cc(M)$, for all $c_1$, $c_2>0$,
with $M=\ell_1''(1)^{-1}(\|\ell_1'''\|_\infty+\|\ell_2'''\|_\infty)$.
It follows that  example \eqref{example} satisfies the hypotheses of Theorem \ref{mainthm}
for all $c_1, c_2>0$, with $c_1+c_2$ sufficiently large.

For  positive definite and symmetric matrices
$\Sigma\in\slt$, we find by the consideration of
eigenvalues that
\[
H_1(\Sigma)-1\sim |\Sigma-I|^2\isp{and}
|H_1(\Sigma)-H_2(\Sigma)|\le C|\Sigma-I|^3,
\]
in a neighborhood of $\Sigma=1$.  Thus, $W(\Sigma(F))-1$ has a relative minimum
when $\Sigma(F)=I$, i.e.\ when $F=\sigma  U$ for some $\sigma>0$ and $U\in\sot$.

More generally, we could take
\[
\Phi(x_1,x_2)=1+(c_1+c_2)(x_1-1)-c_2(x_1-x_2)G(x_1,x_2), \quad c_1, c_2>0.
\]
If   $G:\rr_+^2\to\rr$ is any smooth function with $\|G(\ell_1,\ell_2)\|_{C^3}$  
uniformly bounded
 independent
of $c_1,c_2$, then there exists an $M>0$,
such that $f(u)\in\cc(M)$, for all $c_1,c_2>0$.

\begin{remark}
The range of the map $F\mapsto \left(H_1(\Sigma(F)),H_2(\Sigma(F))\right)$
from the domain $\glt$ into $\rr_+^2$ is given by
\[
\mathcal R=\{(x_1,x_2)\in\rr_+^2: 
\Delta(x)\equiv 3x_1^2x_2^2-4x_1^3-4x_2^3+6x_1x_2-1\ge0\}.
\]
The boundary of this region corresponds to positive definite symmetric matrices 
$\Sigma(F)=(\det A(F))^{-1/3}A(F)$ with a repeated eigenvalue,
and it is given by the curve $\mathcal H$ defined in \eqref{bdry}.
See Figure \ref{region}.
This is because if $A(F)$ has eigenvalues $\{\lambda_i\}_{i=1}^3$ and
\[
(x_1,x_2)=\left(H_1(\Sigma(F)),H_2(\Sigma(F)\right),
\]
  then
   \begin{equation*}
\Delta(x)=\tfrac1{27}(\lambda_1-\lambda_2)^2(\lambda_2-\lambda_3)^2(\lambda_3-\lambda_1)^2/(\lambda_1\lambda_2\lambda_3)^2.
\end{equation*}
This expression is nonnegative and vanishes if and only if $A(F)$ has a repeated eigenvalue.\footnote
{The expression $\Delta(x)$ is the discriminant of the characteristic polynomial of the
shear strain tensor.}
By the spectral theorem, a matrix $A$ is positive definite, symmetric, with a repeated eigenvalue
if and only if $A=F(\lambda,\omega)$, for some $\lambda\in\rr^2_+$, $\omega\in S^2$,
so  $\partial\mathcal R=\mathcal H$, by \eqref{ssinvariants2}, \eqref{ssinvariants3}.

\end{remark}
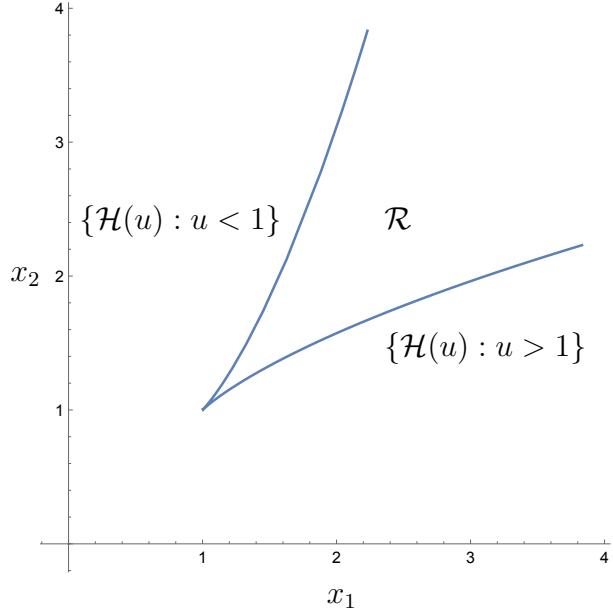
\begin{figure}[h!]
\caption{The region $\mathcal R$ and its boundary $\mathcal H$.}
 \label{region}
\ \\
\setlength\unitlength{1mm}
\begin{center}

\begin{overpic}[scale=.4]
{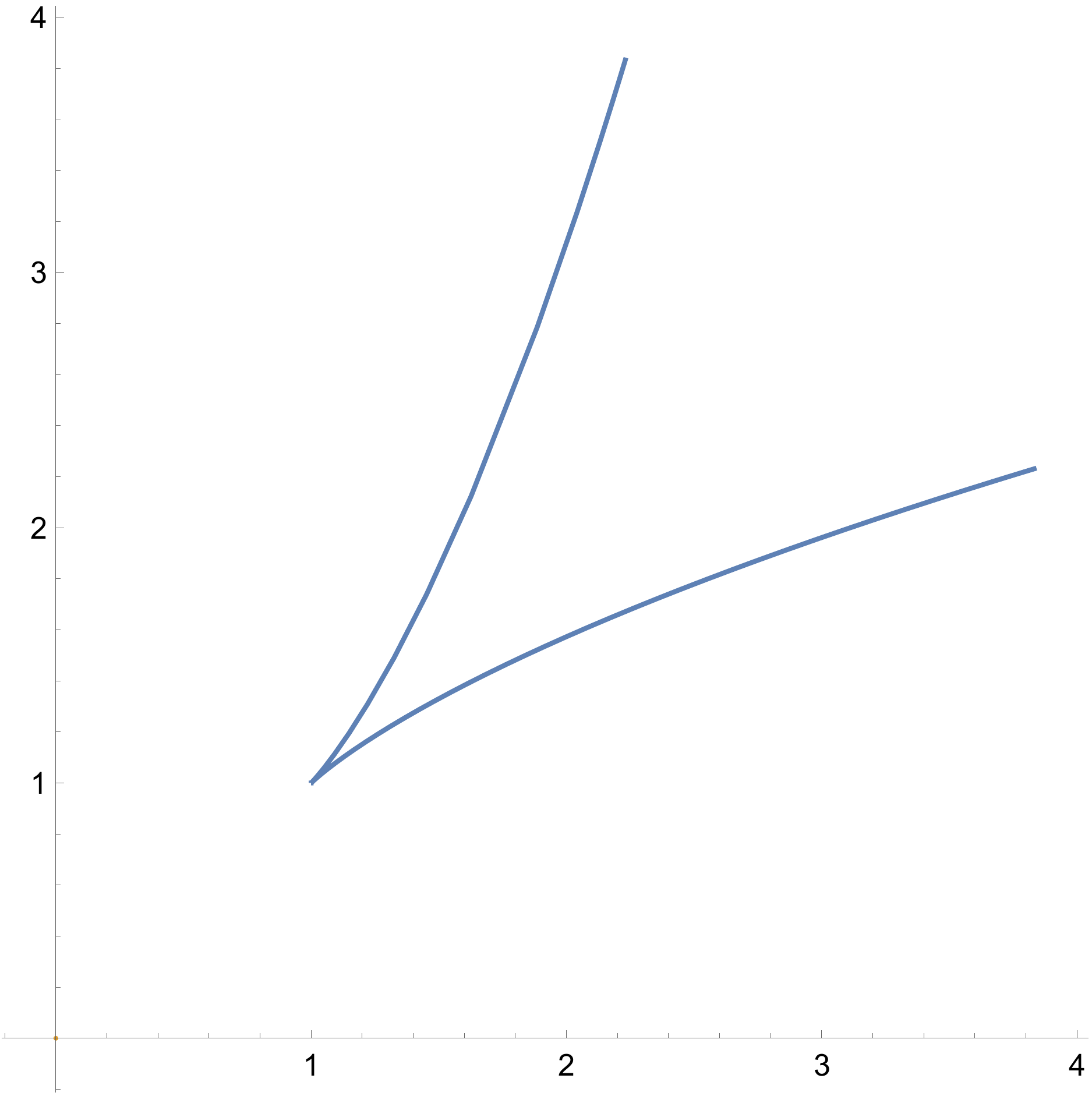}
\put(60,38){$\{\mathcal H(u):u>1\}$}
\put(7,60){$\{\mathcal H(u):u<1\}$}
\put(60,60){$\mathcal R$}
\put(50,-5){$x_1$}
\put(-5,51){$x_2$}
\end{overpic}
\end{center}
\end{figure}

\begin{proposition}
\label{BEineq}
Any $C^2$ strain energy function $W$ satisfying 
\eqref{sef}, \eqref{objective},
\eqref{isotropic}, \eqref{homogeneity}
is consistent with the Baker-Ericksen condition \cite{Baker-Ericksen-1954}
along $\mathcal H$, provided \eqref{sssefrep} holds with
\[
f'(1)=0 \isp{and} f''(u)\ge0,\quad u\in\rr_+.
\]
\end{proposition}

\begin{proof}
As noted in  Lemma \ref{flomdef}, the eigenvalues of $F(\lambda,\omega)$ are
$\lambda_1$, $\lambda_2$ with corresponding eigenspaces $\spn\{\omega\}$,
$\spn\{\omega\}^\perp$.  By \eqref{sscs0},
the Cauchy stress
$T\circ F(\lambda,\omega)$ shares these  eigenspaces, with corresponding
eigenvalues
\[
t_1(\lambda)=\left(\lambda_1\lambda_2^2\right)^{-1}\lambda_1\llll_{,1}(\lambda),\quad
t_2(\lambda)=\left(\lambda_1\lambda_2^2\right)^{-1}\half\lambda_2\llll_{,2}(\lambda).
\]

Since the eigenvalue $\lambda_2$ is repeated, the Baker-Ericksen inequality reads
\[
\frac{t_1(\lambda)-t_2(\lambda)}{\lambda_1-\lambda_2}\ge0,
\]
for all $\lambda=(\lambda_1,\lambda_2)\in\rr_+^2$ with $\lambda_1\ne\lambda_2$.
This is equivalent to the statement
\[
\frac{\lambda_1\llll_{,1}(\lambda)-\half \lambda_2\llll_{,2}(\lambda)}{\lambda_1-\lambda_2}\ge0,
\]
for all $\lambda=(\lambda_1,\lambda_2)\in\rr_+^2$ with $\lambda_1\ne\lambda_2$.

By \eqref{lamudef}, \eqref{ellforms}, we have
\begin{align*}
&\frac{\lambda_1\llll_{,1}(\lambda)-\half \lambda_2\llll_{,2}(\lambda)}{\lambda_1-\lambda_2}\\
&\qquad=\volume^{(\hd-1)/3}u^{-2/3}\;
\frac{\tfrac\hd3(\lambda_1-\lambda_2u)f(u)+(\lambda_1u+\half\lambda_2u^2)f'(u)}
{\lambda_1-\lambda_2}\\
&\qquad=\volume^{(\hd-1)/3}u^{-2/3}\;
\frac{3u^2}2\frac{f'(u)}{u-1},
\end{align*}
for all $\lambda_1,\lambda_2\in\rr_+$ with $\lambda_1\ne\lambda_2$.
This is nonnegative since $f'(1)=0$  
and $f''\ge0$, by assumption.
(Note that if $f''>0$, then the inequality is strict.)
\end{proof}

\begin{remark}
Any function $f\in\cc(M)$ satisfies the 
   satisfies the hypotheses of  Theorem \ref{BEineq} with 
strict inequality 
in the neighborhood $\uu(\delta)$.
\end{remark}

As a final result, we discuss the physical significance of the parameters
$\kk(\hd)$ and $\beta(f)$.  

\begin{lemma}
\label{moduli}
Let $\hd\in\rr\setminus[0,3]$ and $M\ge0$.
Suppose that $W$ satifies \eqref{sef}, \eqref{objective}, \eqref{isotropic}, \eqref{homogeneity},
and for some $f\in\cc(M)$
\begin{equation*}
W\circ F(\lambda,\omega)=\volume^{\hd/3}f(u)
\isp{for all} \lambda\in\rr_+^2,\quad \omega\in S^2.
\end{equation*}
Then the bulk and shear moduli of $W$ at $F=I$ are
\begin{equation*}
\kk=\kk(\hd)=\tfrac\hd3\left(\tfrac\hd3-1\right) \isp{ and }\shear=\tfrac34\beta(f),
\end{equation*}
respectively.

 \end{lemma}

\begin{proof}
The bulk modulus at $F=I$  is defined as the change in  
$-\mathcal P(\alpha)$ with respect to the fractional
change in volume at $\alpha=1$.  
Thus, from  \eqref{respress}, we find that
\begin{equation*}
\kk=\lim_{\alpha\to1}\frac{-\mathcal P(\alpha)+\mathcal P(1)}{\alpha^3-1}
=\lim_{\alpha\to1}\frac{\frac\hd3\left (\alpha^{\hd-3}-1\right)}{\alpha^3-1} =\kk(\hd).
\end{equation*}

For isotropic materials, the linearization of the operator
\[
D\cdot S(D_y\varphi)
\]
at $\varphi(y)=y$ is
\begin{equation}
\label{linearelas}
\shear\Delta\bar\varphi+\left(\kk+\tfrac13\;\shear\right)\nabla(\nabla\cdot\bar\varphi),\quad
\bar\varphi(y)=\varphi(y)-y,
\end{equation}
where $\shear$ is the shear modulus.
In the case of spherically symmetric deformations, we have from \eqref{d2phi} that
\[
\Delta\varphi(y)=\nabla(\nabla\cdot\varphi(y))=\left(\phi''(r)+\tfrac2r\phi'(r)-\tfrac2{r^2}\phi(r)\right)\omega,
\]
and so \eqref{linearelas} reduces to
\begin{equation*}
\left(\kk+\tfrac43\;\shear\right)\left({\bar\phi}''(r)+\tfrac2r{\bar\phi}'(r)-\tfrac2{r^2}\bar\phi(r)\right)\omega,
\quad\bar\phi(r)=\phi(r)-r.
\end{equation*}
Comparing this with \eqref{ODE2.0}, we obtain
\[
\kk+\tfrac43\;\shear=U_1(1)=\half U_2(1)=\kk(\hd)+\beta(f),
\]
which implies that $\shear=\tfrac34\beta(f)$.
\end{proof}

\bibliography{AffineElasticity}
\bibliographystyle{acm}

\end{document}